\documentclass[11pt]{article}

\usepackage{mathrsfs}
\usepackage{enumerate}
\usepackage[section]{algorithm}
\usepackage{algorithmic}
\usepackage{multirow}
\usepackage{xcolor}

\usepackage{graphicx}
\usepackage{amsmath,amsfonts,amsthm,amsbsy,amssymb}
\usepackage{fixmath}
\usepackage{amssymb,latexsym}

\usepackage{multirow} 

\usepackage{hyperref}

\usepackage{cleveref}

\def\by{\pmb{y}}

\def\bbC{\mathbb{C}}
\def\bbI{\mathbb{I}}

\def\bbR{\mathbb{R}}

\def\scrG{\mathscr{G}}

\def\cP{{\cal P}}

\def\cR{{\cal R}}
\def\cX{{\cal X}}
\def\cY{{\cal Y}}

\def\wtd{\widetilde}
\def\what{\widehat}

\usepackage{accents}

\newcommand\STM[2]{{\rm St}(#1,#2)}

\DeclareMathOperator{\diag}{diag}
\DeclareMathOperator{\dist}{dist}
\DeclareMathOperator{\eig}{eig}

\DeclareMathOperator{\rank}{rank}

\DeclareMathOperator{\tr}{tr}
\DeclareMathOperator{\ufd}{ufd}

\DeclareMathOperator{\F}{F}
\DeclareMathOperator{\HH}{H}
\DeclareMathOperator{\T}{T}
\DeclareMathOperator{\UI}{ui}

\DeclareMathOperator{\KKT}{KKT}

\newtheorem{theorem}{Theorem}[section]
\newtheorem{lemma}{Lemma}[section]

\theoremstyle{definition}

\newtheorem{remark}{Remark}[section]

\allowdisplaybreaks

\numberwithin{equation}{section}
\numberwithin{figure}{section}

\title{Convergence Analysis of Two Alternating Iterative Schemes for Tucker Decomposition}

\author{Ren-Cang Li%
\thanks{Department of Mathematics, University of Texas at Arlington, Arlington, TX 76019-0408, USA.
        Supported in part by NSF DMS-2407692.
        Email: {\tt rcli@uta.edu}.}
\and
Li Wang%
\thanks{Department of Mathematics, University of Texas at Arlington, Arlington, TX 76019-0408, USA.
        Supported in part by NSF DMS-2407692.
        Email: {\tt li.wang@uta.edu}.}
\and
Mei Yang%
\thanks{Department of Mathematics, University of Texas at Arlington, Arlington, TX 76019-0408, USA.
        Email: {\tt mei.yang@uta.edu}.}
}

\date{
    May 15, 2026
}

\begin{document}

\maketitle

%

\begin{abstract}
The higher-order orthogonal iteration (HOOI) and the alternating subspace iteration (ASI)
are two popular numerical methods for computing the Tucker decomposition of a multiple-mode tensor.
Xu [{\em Linear and Multilinear Algebra}, 66(11):2247--2265, 2018] proposed a variation of HOOI, called
the {\em greedy\/} HOOI, which has an extra alignment action between consecutive approximations.
Kroonenberg and De Leeuw [{\em Psychometrika}, 45(1):69--97, 1980] analyzed
the convergence of ASI but their analysis has gaps. These analysis were for a real tensor only.
In this paper, we present detailed convergence analysis of the two methods that is applicable
to a complex tensor with a real tensor being a special case,
and it is shown both methods are globally convergent to stationary points under mild conditions
while the objective function monotonically increases. Numerical examples are presented to demonstrate
the convergence behavior of the methods.


\end{abstract}

\clearpage
{\scriptsize
\tableofcontents
}

\section{Introduction}\label{sec:intro}
A matrix is sometimes referred as a 2-mode or 2-way array and in general a tensor is an array of more than 2 modes or 2 ways \cite{bako:2025,dldv:2000a,elde:2007,koba:2009}. In various applications, data is naturally collected as tensors that need to be
analyzed to gain knowledge. Most tensor decompositions of 3 modes or more are extensions
from corresponding ones in
matrix analysis and can only be done approximately because generically these tensor decompositions
simply do not exist in theory. The goal of this paper is to present detailed convergence analysis
for two alternating schemes, the higher-order orthogonal iteration (HOOI) \cite{dldv:2000b} and the alternating subspace iteration (ASI) \cite{krdl:1980},  to compute the so-called Tucker decomposition (TD).
Our presentation is for complex tensors for generality, but it equally applies to real tensors
straightforwardly.


Let $B\equiv[b_{i_1i_2\cdots i_m}]\in\bbC^{n_1\times n_2\times\cdots\times n_m}$ be an $m$-mode tensor. We will adopt a few operations on tensors: the Frobenius norm, unfolding, and mode-multiplication, as introduced in \cite{dldv:2000a,elde:2007}.  The tensor Frobenius norm of $B$ is
$$
\|B\|_{\F}=\Big(\sum_{i_1,i_2,\ldots,i_m}|b_{i_1i_2\cdots i_m}|^2\Big)^{1/2}.
$$
Let $N=n_1n_2\cdots n_m$.
There are $m$ different unfolding matrices from  $m$-mode tensor $B$. Following the informal definition
\cite[Definition~I.8]{bako:2025} and denoting by $B_{\ufd,\ell}\in\bbC^{n_{\ell}\times N/n_{\ell}}$ the $\ell$th unfolding matrix,
we define the $(i_1\cdots i_{\ell-1}\,i_{\ell+1}\cdots i_m)$th column of $B_{\ufd,\ell}$ as
\begin{equation}\label{eq:ufd,ell}
\begin{bmatrix}
  b_{i_1\cdots i_{\ell-1}\,1\,i_{\ell+1}\cdots i_m} \\
  b_{i_1\cdots i_{\ell-1}\,2\,i_{\ell+1}\cdots i_m} \\
  \vdots \\
  b_{i_1\cdots i_{\ell-1}\,n_{\ell}\,i_{\ell+1}\cdots i_m},
\end{bmatrix}
\end{equation}
i.e., the $\ell$th subscript runs from $1$ to $n_{\ell}$ while all others stay fixed. As each subscript $i_j$ $(j\ne\ell$)
varies from $1$ to $n_j$, $B_{\ufd,\ell}$ has a total of $N/n_{\ell}$ columns. There is a question of how these columns
should be ordered to form $B_{\ufd,\ell}$. For the interest of this paper, i.e., computing the Tucker decomposition,
the order of arranging these columns does not matter so long it is done in the same way every time
the $\ell$th unfolding matrix is formed.
The $\ell$th mode-multiplication by $B$ is defined as
\begin{equation}\label{eq:j-times}
(B\times_{\ell} X)_{(i_1i_2\cdots i_m)}=\sum_{j=1}^{n_{\ell}}x_{i_{\ell}\,j}b_{i_1\cdots i_{\ell-1}\,j\,i_{\ell+1}\cdots i_m}
\quad\mbox{for $X\equiv[x_{ij}]\in\bbC^{k\times n_{\ell}}$},
\end{equation}
to yield $B\times_{\ell} X\in\bbC^{n_1\times\cdots\times n_{\ell-1}\times k\times n_{\ell+1}\times\cdots\times n_m}$,
where $(\,\cdot\,)_{(i_1i_2\cdots i_m)}$ refers to the $(i_1,i_2,\ldots, i_m)$th entry of a tensor.
Denote by
$$
\STM{k}{n}=\{P\in\bbC^{n\times k}\,:\,P^{\HH}P=I_k\}\subset\bbC^{n\times k},
$$
the complex Stiefel manifold, where $1\le k\le n$.

Given integers $1\le k_{\ell}\le n_{\ell}$ for $1\le \ell\le m$,
computing an approximate TD for tensor $B$ refers to finding $P_{\ell}\in\STM{k_{\ell}}{n_{\ell}}$ for $1\le \ell\le m$ and a tensor
$T\in\bbC^{k_1\times k_2\times\cdots\times k_m}$, called {\em the core tensor}, such that
\begin{equation}\label{eq:hatB}
B\approx \what B:=T\times_1 P_1\times_2 P_2\cdots\times_m P_m
\end{equation}
optimally in the sense that $\|B-\what B\|_{\F}$ is minimized over $(P_1,P_2,\ldots,P_m)$ subject to  $P_{\ell}\in\STM{k_{\ell}}{n_{\ell}}\,\,\,\forall\ell$.
The optimal $\what B$ in \eqref{eq:hatB} is also called the best rank-$(k_1, k_2,\ldots,k_m)$ approximation of tensor $B$
in \cite{dldv:2000b}.
Two commonly used methods for computing TD are the higher-order orthogonal iteration (HOOI) \cite{dldv:2000b}
and the alternating subspace iteration (ASI)~\cite{krdl:1980}. Despite their wide uses, their convergence analysis
are incomplete. Xu~\cite{xu:2018} analyzed the convergence of a HOOI variant, called called
the {\em greedy\/} HOOI, which has an extra alignment action between consecutive approximations, while
Kroonenberg and De Leeuw~\cite{krdl:1980} analyzed the convergence of ASI but their arguments has technical gaps.
More detailed discussions will come at the ends of in \cref{sec:HOOI,sec:ASI}.

TD~\eqref{eq:hatB} for $m=3$ corresponds to the one in Tucker~\cite{tuck:1966} in which he proposed to truncate the HOSVD (high order SVD
as known today \cite{dldv:2000a}) to compute $P_{\ell}$ for $\ell=1,2,3$. The earlier formulation in \cite{tuck:1963} is more general and does not
impose orthonormality on each $P_{\ell}$. However,  as far as the approximation error $\|B-\what B\|_{\F}$ is concerned, the two formulations are
equivalent because we can always orthonormalize $P_{\ell}$ for $\ell=1,2,3$, if they are not orthonormal yet, and update $T$ properly at the same time. The invention of
TD was clearly driven by the need to better analyze 3-mode data array from psychology \cite{tuck:1963,tuck:1964,tuck:1966,tuck:1972}.
How to choose suitable $k_{\ell}$
is nontrivial and data-dependent.
The issue is not the focus of this paper, however, and interested readers may consult \cite{kidk:2003} for some idea.

The rest of this paper is organized as follows. In \cref{sec:angle-space}, we collect some basic numerical linear algebra
terminology and results that will be used for our convergence analysis. \Cref{sec:TD} reformulates minimizing
$\|B-\what B\|_{\F}$ equivalently into a maximization problem which is further reformulated into $m$ different ways by expressing its objective
function in terms of one of the $m$ different unfolding matrices, paving the way for solving the maximization problem
alternatingly over all $P_{\ell}$. In \cref{sec:HOOI}, the higher-order orthogonal iteration (HOOI) is explained, along
with its implementation suggestions and convergence analysis, and in \cref{sec:ASI}, we explain
the alternating subspace iteration (ASI) and present its convergence analysis. Numerical experiments to demonstrate
the two algorithms are presented in \cref{sec:egs}. Finally we draw our conclusions in \cref{sec:conclu}.

{\bf Notation.}
We follow the following notation convention throughout this paper.
  $\bbC^{m\times n}$  is the set of $m\times n$ complex matrices,  $\bbC^n=\bbR^{n\times 1}$, and $\bbC=\bbR^1$,
        and similarly $\bbR^{m\times n}$,  $\bbR^n$, and $\bbR$ except for the real numbers.
$I_n\in\bbR^{n\times n}$ is the identity matrix or simply $I$ if its size is clear from the context.
For a matrix/vector $X$,
  $X^{\T}$ and $X^{\HH}$ stand for its transpose and complex conjugate transpose, respectively.
For a matrix $X\in\bbC^{m\times n}$,
$\|X\|_2$, $\|X\|_{\F}$, and $\|X\|_{\tr}$
are its spectral norm, Frobenius norm, and trace norm (also known as the nuclear norm), respectively, and
$\sigma_{\min}(X)$ is the smallest singular
value\footnote {It is understood that $X$ has $\min\{m,n\}$
     singular values.} of $X$.
$\cR(X)$ is the column space of a matrix $X$, spanned by its columns.
For a square matrix $X\in\bbC^{n\times n}$, $\tr(X)$ is its trace, i.e., the sum of its diagonal entries.
For a Hermitian matrix $A\in\bbC^{n\times n}$, $\eig(A)=\{\lambda_i(A)\}_{i=1}^n$ denotes the multiset of its
        eigenvalues (counted by multiplicities)
        arranged in the decreasing order:
        $$
        \lambda_1(A)\ge\lambda_2(A)\ge\cdots\ge\lambda_n(A).
        $$
  A matrix $A\succ 0\, (\succeq 0)$ means that it is Hermitian and positive definite (semi-definite), and
        accordingly
        $A\prec 0\, (\preceq 0)$ if $-A\succ 0\, (\succeq 0)$.

\section{Preliminaries}\label{sec:angle-space}
To serve the rest of this paper, we now introduce the canonical angles between two subspaces of equal dimension,
error bounds on extreme eigenvalue approximation of a Hermitian matrix and on extreme singular
value approximation of a matrix, all via matrix traces.
They are known results.

Let\footnote {If $k=n$, everything in this section is trivially true.}
$1\le k<n$.
The collection of all $k$-dimensional subspaces in $\bbC^n$ endowed with an appropriate metric is the
Grassmann manifold $\scrG_k(\bbC^n)$.
Let
$\cX=\cR(X)$ and $\cY=\cR(Y)$ be two points in $\scrG_k(\bbC^n)$,
where $X,\,Y\in\STM{k}{n}$. The canonical angles
$\theta_1(\cX,\cY)\ge\cdots\ge\theta_k(\cX,\cY)$ between $\cX$ and $\cY$ are defined as \cite{stsu:1990}
$$
	0\le\theta_i\equiv\theta_i(\cX,\cY):=\arccos \sigma_i(X^{\HH}Y)\le\frac {\pi}2 \quad\mbox{for $1\le i\le k$},
$$
and accordingly, $\Theta(\cX,\cY)=\diag(\theta_1,\dots,\theta_k)\in\bbC^{k\times k}$ is
the diagonal matrix of the canonical angles between $\cX$ and $\cY$.
It is known that
\begin{equation}\label{eq:sinTheta:UI}
\dist_{\UI}(\cX,\cY):=\|\sin\Theta(\cX,\cY)\|_{\UI}
\end{equation}
is a unitarily invariant metrics on the
Grassmann manifold
$\scrG_k(\bbC^n)$  \cite[p.94]{stsu:1990} for any
unitarily invariant norm $\|\cdot\|_{\UI}$ \cite{stsu:1990}, such as the commonly used ones: $\|\cdot\|_2$, $\|\cdot\|_{\F}$ and $\|\cdot\|_{\tr}$, that are generic with respect to matrix sizes.
In particular,
\begin{subequations}\label{eq:sinTheta:UI=2,F}
\begin{align}
\dist_2(\cX,\cY)&:=\sin\theta_1=\|XX^{\HH}-YY^{\HH}\|_2, \label{eq:sinTheta:UI=2} \\
\dist_{\F}(\cX,\cY)&:=\Big[\sum_{i=1}^k\sin^2\theta_i\Big]^{1/2}=\frac 1{\sqrt 2}\|XX^{\HH}-YY^{\HH}\|_{\F},  \label{eq:sinTheta:UI=F}
\end{align}
\end{subequations}
where the last equalities in both \eqref{eq:sinTheta:UI=2} and \eqref{eq:sinTheta:UI=F} are a corollary of
\cite[Theorem~5.5]{stsu:1990}.

The orthonormal basis matrix $X$ of $\cX$ is not unique, and neither is $Y$ of $\cY$. For that reason, it is not expected
that $\|X-Y\|_{\F}$ be in the order of $\Theta(\cX,\cY)$. In particular, possibly $\|X-Y\|_{\F}>0$ even in the case
$\cX=\cY$ unless the basis matrices $X$ and $Y$ are judicially chosen. One way to do this is to fix one of them, say $Y$, and then
find a new basis matrix $\wtd X\in\STM{k}{n}$ of $\cX$
to align $X$ with $Y$
according to $\Theta(\cX,\cY)$. This is done in the next lemma.

\begin{lemma}\label{lm:basis-align}
Let $\cX=\cR(X),\,\cY=\cR(Y)\in\scrG_k(\bbC^n)$ where $X,\,Y\in\STM{k}{n}$, and let $\wtd X:=XQ_*$ where $Q_*\in\STM{k}{n}$ is
an orthonormal polar factor of $X^{\HH}Y$.
Then
\begin{equation}\label{eq:basis-align}
\|\wtd X-Y\|_{\UI}=2\Big\|\sin\frac {\Theta(\cX,\cY)}2\Big\|_{\UI},
\end{equation}
and\footnote {A sharper inequality than the second inequality in \eqref{eq:basis-align:equiv} is
    $$
    \|\wtd X-Y\|_{\UI}
    \le \big\|\sin\Theta(\cP_1,\cP_2)\big\|_{\UI}\Big(1
        +(\sqrt 2-1)\big\|\sin\Theta(\cP_1,\cP_2)\big\|_2^2\Big),
    $$
    upon using $\sin\frac 12\theta\le\frac 12\sin\theta+\frac {\sqrt 2-1}2\sin^3\theta$
    for $0\le\theta\le\pi/2$ \cite[p.14]{wawl:2025}.
    }
\begin{equation}\label{eq:basis-align:equiv}
\|\sin\Theta(\cX,\cY)\|_{\UI}\le\|\wtd X-Y\|_{\UI}\le\sqrt 2\|\sin\Theta(\cX,\cY)\|_{\UI}.
\end{equation}
\end{lemma}

\begin{proof}
Since  $Q_*\in\STM{k}{n}$ is an orthonormal polar factor of $X^{\HH}Y$, $X^{\HH}Y$ admits a polar decomposition
$X^{\HH}Y=Q_*\Gamma$ where $\Gamma\succeq 0$. It is clear that the eigenvalues of $\Gamma$
are $\cos\theta_i$ for $1\le i\le k$ where for convenience we write $\theta_i=\theta_i(\cX,\cY)$. Notice that
$$
[XQ_*-Y]^{\HH}[XQ_*-Y]=2I_k-2\Gamma
$$
whose eigenvalues are
$$
2(1-\cos\theta_i)=4\sin^2\frac {\theta_i}2\quad\mbox{for $1\le i\le k$}.
$$
As a consequence, the singular values of $XQ_*-Y$ are $2\sin\frac {\theta_i}2$ for $1\le i\le k$, yielding \eqref{eq:basis-align}.
Finally notice, for $0\le\theta\le\frac {\pi}2$,
$$
\sqrt 2\sin\frac 12\theta\le\sin\theta=2\sin\frac 12\theta\,\cos\frac 12\theta\le 2\sin\frac 12\theta
$$
yielding $\sin\theta\le 2\sin\frac 12\theta\le\sqrt 2\sin\theta$, leading to
the inequalities in \eqref{eq:basis-align:equiv} from \eqref{eq:basis-align}.
\end{proof}


Next we state error bounds on extreme eigenvalue approximation via matrix trace.
\Cref{lm:maxtrace-proj} was taken from
\cite{li:2026}. It is noted that the second inequality in \eqref{eq:eig2max} is
\cite[Theorem~1]{kove:1991} and can also be derived from some of the estimates in
\cite[Chapter~3]{wein:1974} and by a minor modification to the proof of \cite[Theorem~2.2]{li:2004c}.

\begin{lemma}[{\cite[Theorem~4.1]{li:2026}}]\label{lm:maxtrace-proj}
Let $H\in\bbC^{n\times n}$ be Hermitian and $P_*\in\STM{k}{n}$ whose column space $\cR(P_*)$ is the invariant subspace
of $H$ associated with its $k$ largest eigenvalues. Suppose that $\lambda_k(H)-\lambda_{k+1}(H)>0$.
Given $P\in\STM{k}{n}$, let
$$
\eta=\tr(P_*^{\HH}HP_*)-\tr(P^{\HH}HP), \quad \epsilon=\sqrt{\frac {\eta}{\lambda_k(H)-\lambda_{k+1}(H)}}.
$$
Then\footnote {The first inequality in \eqref{eq:eig2max} actually holds so long as
      $\cR(P_*)$ is a $k$-dimensional invariant subspace of $H$. It is the second inequality
      that needs the condition of $\cR(P_*)$ being associated with the $k$ largest eigenvalues of $H$.}
\begin{equation}\label{eq:eig2max}
\frac {\|HP-P(P^{\HH}HP)\|_{\F}}{\lambda_1(H)-\lambda_n(H)}\le\|\sin\Theta(\cR(P),\cR(P_*))\|_{\F}\le {\epsilon}.
\end{equation}
\end{lemma}

For extreme singular value approximation via matrix trace,
the next lemma \cite[Lemma~4.2]{li:2026} are likely well-known. For example, it (for the real number field) is implied in the discussion in \cite{wazl:2022a} before
\cite[Lemma~ 3.2]{wazl:2022a} there.

\begin{lemma}\label{lm:polar2max}
Let $B\in\bbC^{n\times k}$.
\begin{enumerate}[{\rm (a)}]
  \item $\tr(P^{\HH}B)\le\|B\|_{\tr}$ for any $P\in\STM{k}{n}$;
  \item $\tr(P^{\HH}B)=\|B\|_{\tr}$ where $P\in\STM{k}{n}$ if and only if $B=P\Lambda$ with $\Lambda\succeq 0$;
  \item We have
        $$
        \max_{P\in\STM{k}{n}}\tr(P^{\HH}B)=\|B\|_{\tr}
        $$
        and the optimal value $\|B\|_{\tr}$ is achieved by any orthonormal polar factor $P_*$ of $B$.
\end{enumerate}
\end{lemma}

Lemma~\ref{lm:polar2max} says that $\tr(P^{\HH}B)$  is bounded  above by $\|B\|_{\tr}$ always and the upper bound
$\|B\|_{\tr}$ is achieved
by any orthonormal polar factor $P_*$ of $B$ and also any maximizer of $\tr(P^{\HH}B)$ over $P\in\STM{k}{n}$ is an orthonormal polar factor of $B$.
For numerical computation, an orthonormal polar factor of $B$ can be constructed from the thin SVD $B=U\Sigma V^{\HH}$ as
$P_*=UV^{\HH}$. The orthonormal polar factor of $B$ is unique if and only if $\rank(B)=k$ \cite{li:1993b,li:2014HLA}.
Conceivably,
the closer $\tr(P^{\HH}B)$ is to the upper bound, the closer $P$ approaches to an orthonormal polar factor of $B$.
The results of the next lemma quantify the last statement.

\begin{lemma}[{\cite[Theorem~4.1]{li:2026}}]\label{lm:polar2max'}
Let $B\in\bbC^{n\times k}$ and suppose $\rank(B)=k$. Let $P_*$ be the
orthonormal polar factor of $B$.  Given $P\in\STM{k}{n}$, let
$$
\eta=\|B\|_{\tr}-\tr(P^{\HH}B), \quad
\epsilon=\sqrt{\frac {2\eta}{\sigma_{\min}(B)}}.
$$
Then
\begin{equation}\label{eq:polar2max}
\frac {\|B-P(P^{\HH}B)\|_{\F}}{\|B\|_2}\le\|\sin\Theta(\cR(P),\cR(P_*))\|_{\F}\le {\epsilon}\,.
\end{equation}
\end{lemma}


\section{Tucker Model}\label{sec:TD}
In the rest of this paper, for integer tuple $(\ell,i_1,i_2,\ldots,i_{m-1})$ appears in the same paragraph, it is understood that it is a permutation of $(1,2,\ldots,m)$ such that $i_1<i_2<\cdots<i_{m-1}$. To that specification,
we know  in fact
$$
(\ell,i_1,i_2,\ldots,i_{m-1})=(\ell,1,\ldots,\ell-1,\ell+1,\ldots,m).
$$
For example, in the case, $m=3$,
there are only three such tuples: $(1,2,3)$, $(2,1,3)$, and $(3,1,2)$.

It well-known that minimizing $\|B-\what B\|_{\F}$ is equivalent to solving (see, e.g., \cite{koba:2009})
\begin{equation}\label{eq:opt-TD}
\max_{P_i\in\STM{k_i}{n_i}\,\,\forall i} \,
   \Big\{f(P_1,P_2,\ldots,P_m):=\|B\times_1 P_1^{\HH}\times_2 P_2^{\HH}\cdots\times_m P_m^{\HH}\|_{\F}^2\Big\}.
\end{equation}
Once an optimal $(P_1,P_2,\ldots,P_m)$ is computed, the corresponding core tensor $T$ can be recovered by
$T=B\times_1 P_1^{\HH}\times_2 P_2^{\HH}\cdots\times_m P_m^{\HH}$.

Define
\begin{equation}\label{eq:Ci}
C_{\ell}(P_{i_1},\ldots,P_{i_{m-1}})
    :=\big[B\times_{i_1} P_{i_1}^{\HH}\cdots\times_{i_{m-1}} P_{i_{m-1}}^{\HH}\big]_{\ufd,\ell}
    \in\bbC^{n_{\ell}\times (k_{i_1}\cdots k_{i_{m-1}})}
\end{equation}
for any $(\ell,i_1,i_2,\ldots,i_{m-1})$ as specified moments ago.
Through unfolding, we can express $f(P_1,P_2,\ldots,P_m)$ into $m$ different ways: for each $1\le\ell\le m$
\begin{equation}\label{eq:obj-TD}
f(P_1,P_2,\ldots,P_m)=\big\|P_{\ell}^{\HH}\,C_{\ell}(P_{i_1},\ldots,P_{i_{m-1}})\big\|_{\F}^2
    \equiv\tr\big(P_{\ell}^{\HH}H_{\ell}(P_{i_1},\ldots,P_{i_{m-1}})P_{\ell}\big),
\end{equation}
where
\begin{equation}\label{eq:Hi}
H_{\ell}(P_{i_1},\ldots,P_{i_{m-1}}):=[C_{\ell}(P_{i_1},\ldots,P_{i_{m-1}})][C_{\ell}(P_{i_1},\ldots,P_{i_{m-1}})]^{\HH}\in\bbC^{n_{\ell}\times n_{\ell}}.
\end{equation}
We make an important observation: both $f(P_1,P_2,\ldots,P_m)$ and $H_{\ell}(P_{i_1},\ldots,P_{i_{m-1}})$ for $1\le \ell\le m$ are unitarily invariant in the sense that
\begin{subequations}\label{eq:f:Hi-UI}
\begin{align}
f(P_1Q_1,P_2Q_2,\ldots,P_mQ_m)&\equiv f(P_1,P_2,\ldots,P_m), \label{eq:f-UI} \\
H_{\ell}(P_{i_1}Q_{i_1},\ldots,P_{i_{m-1}}Q_{i_{m-1}})&\equiv H_{\ell}(P_{i_1},\ldots,P_{i_{m-1}})\quad\mbox{for $1\le \ell\le m$}, \label{eq:Hi-UI}
\end{align}
\end{subequations}
for any $Q_{\ell}\in\STM{k_{\ell}}{k_{\ell}}\,\,\forall\, \ell$.
As a consequence, each maximizer tuple
$(P_1,P_2,\ldots,P_m)$ is really a representative from the orbit
$$
\{(P_1Q_1,P_2Q_2,\ldots,P_mQ_m)\,:\,Q_{\ell}\in\STM{k_{\ell}}{k_{\ell}}\,\,\forall 1\le \ell\le m\}
$$
of maximizer tuples, implying only $\cR(P_{\ell})$ for $1\le \ell\le m$ are important. If some $k_{\ell}=n_{\ell}$, then $f$ is invariant
with respect to corresponding $P_{\ell}$ and optimizing over $P_{\ell}$ can be safely ignored by simply taking $P_{\ell}=I_{n_{\ell}}$.

With notation introduced so far, the KKT condition
of \eqref{eq:opt-TD} is
\begin{equation}\label{eq:KKT-TD}
H_{\ell}(P_{i_1},\ldots,P_{i_{m-1}})P_{\ell}=P_{\ell}\Lambda_{\ell},\,\,\Lambda_{\ell}\in\bbC^{k\times k}
\quad\mbox{for $1\le \ell\le m$},
\end{equation}
which is a system of $m$ nonlinear matrix equations in $P_1,\ldots,P_m$ because
$\Lambda_{\ell}$ can be eliminated as follows: pre-multiplying each equation leads to $\Lambda_{\ell}=P_{\ell}^{\HH}H_{\ell}(P_{i_1},\ldots,P_{i_{m-1}})P_{\ell}$, which also implies
$\Lambda_{\ell}\succeq 0$.
There are two ways to look at the $m$ equations in \eqref{eq:KKT-TD}: 1) as a system of coupled NEPv
(nonlinear eigenvalue problem with eigenvector dependency) because each equation is a partial eigen-decomposition of
$H_{\ell}(\cdots)$ evaluated at a solution, and
2) as a system of coupled NPDo (nonlinear polar decomposition
with orthonormal polar factor dependency) \cite{li:2024} because $\Lambda_{\ell}\succeq 0$ always and thus the right-hand side
is the polar decomposition of the matrix-valued function in the left-hand side
that depends on all $m$ orthonormal polar factors. We can use the two ways to
understand the two existing numerical schemes to be detailed in the next two sections, but likely they
were not created by the aforementioned two ways to view the KKT condition \eqref{eq:KKT-TD} in the first place.


\begin{theorem}\label{thm:opt-TD-necessary}
Let $(P_{*1},P_{*2},\ldots,P_{*m})$ be a maximizer of \eqref{eq:opt-TD}. Then equations in \eqref{eq:KKT-TD} hold with
$(P_1,P_2,\ldots,P_m)=(P_{*1},P_{*2},\ldots,P_{*m})$ and $\Lambda_{*\ell}:=P_{*\ell}^{\HH}H_{\ell}(P_{*i_1},\ldots,P_{*i_{m-1}})P_{*\ell}$,
moreover, $\eig(\Lambda_{*\ell})$ consists of the $k_{\ell}$ largest eigenvalues of $H_{\ell}(P_{*i_1},\ldots,P_{*i_{m-1}})$.
\end{theorem}

\begin{proof}
That equations in \eqref{eq:KKT-TD} are satisfied because any maximizer must be a stationary point and hence satisfies the KKT condition.
Suppose, to the contrary, that $\eig(\Lambda_{*\ell})$ did not consist of the $k_{\ell}$ largest eigenvalues of $H_{\ell}(P_{*i_1},\ldots,P_{*i_{m-1}})$.
Let $\what P_{\ell}\in\bbC^{n_{\ell}\times k_{\ell}}$ be an orthonormal basis matrix associated with the $k_{\ell}$ largest eigenvalues of $H_{\ell}(P_{*i_1},\ldots,P_{*i_{m-1}})$.
We have
\begin{align*}
f(P_{*1},\ldots,P_{*\,\ell-1},,\what P_{\ell},P_{*\,\ell+1},\ldots,P_{*m})
   &=\tr(\what P_{\ell}^{\HH}H_{\ell}(P_{*i_1},\ldots,P_{*i_{m-1}})\what P_{\ell}) \\
   &>\tr(P_{*\ell}^{\HH}H_{\ell}(P_{*i_1},\ldots,P_{*i_{m-1}})P_{*\ell})\\
   &=f(P_{*1},P_{*2},\ldots,P_{*\,\ell-1}, P_{*\ell},P_{*\,\ell+1},\ldots,P_{*m}),
\end{align*}
contradicting that $(P_{*1},P_{*2},\ldots,P_{*m})$ is a maximizer of \eqref{eq:opt-TD}.
\end{proof}

In the next two sections, we will discuss two alternating iterative schemes to  compute an optimal TD
through numerically solving \eqref{eq:opt-TD}. Given an initial guess
$(P_1^{(0)},P_2^{(0)},\ldots,P_m^{(0)})$ where $P_{\ell}^{(0)}\in\STM{k_{\ell}}{n_{\ell}}\,\,\,\forall \ell$, each scheme generates
a sequence of approximations
$$
\big\{(P_1^{(j)},P_2^{(j)},\ldots,P_m^{(j)})\big\}_{j=0}^{\infty}
$$
following
the flow of the Gauss-Seidel-type updating:
for $j=0,1,2,\ldots$
\begin{align}
    (P_1^{(j)},P_2^{(j)},\ldots,P_m^{(j)})
&\to (P_1^{(j+1)},P_2^{(j)},\ldots,P_m^{(j)}) \nonumber\\
&\to (P_1^{(j+1)},P_2^{(j+1)},P_3^{(j)},\ldots,P_m^{(j)}) \nonumber\\
&\to \cdots \nonumber\\
&\to (P_1^{(j+1)},P_2^{(j+1)},\ldots,P_m^{(j+1)}). \label{eq:flow-GS}
\end{align}
For the sake of stating the alternating iterative schemes and their convergence analysis, we introduce,
for $\ell=1,2,\ldots,m$,
\begin{subequations}\label{eq:CijHij}
\begin{align}
C_{\ell}^{(j)}&=C_{\ell}(P_1^{(j+1)},\ldots,P_{\ell-1}^{(j+1)},P_{\ell+1}^{(j)},\ldots, P_m^{(j)})\in\bbC^{n_{\ell}\times N/n_{\ell}},
       \label{eq:CijHij-1}\\
H_{\ell}^{(j)}&=[C_{\ell}^{(j)}][C_{\ell}^{(j)}]^{\HH}\in\bbC^{n_{\ell}\times n_{\ell}},
               \label{eq:CijHij-2} \\
\Lambda_{\ell}^{(j)}&=\big[P_{\ell}^{(j)}\big]^{\HH}H_{\ell}^{(j)}P_{\ell}^{(j)}\in\bbC^{k_{\ell}\times k_{\ell}},
               \label{eq:CijHij-3}
\end{align}
\end{subequations}
where $C_{\ell}(\cdots)$ is defined as in \eqref{eq:Ci}.
It is understood, by the convention, that  the first $\ell-1$ arguments in \eqref{eq:CijHij-1} for $\ell=1$ are null, i.e.,
$C_1^{(j)}=C_1(P_2^{(j)},\ldots,P_m^{(j)})$
and similarly $C_m^{(j)}=C_m(P_1^{(j+1)},\ldots,P_{m-1}^{(j+1)})$.

\section{Higher-order Orthogonal Iteration (HOOI)}\label{sec:HOOI}
The higher-order orthogonal iteration (HOOI) \cite[Algorithm~4.2]{dldv:2000b} (see, also \cite[p.478]{koba:2009}) is a commonly used alternating iterative scheme to compute
an optimal Tucker decomposition (TD) of a multiple-mode tensor.
HOOI is a Gauss-Seidel type updating scheme to update $P_{\ell}$ in order via the SCF (self-consistent-field) iteration on
each equation in the KKT condition \eqref{eq:KKT-TD}.
In this section, we will use \Cref{lm:maxtrace-proj} to perform a convergence analysis. It is shown that HOOI will converges
globally to a stationary point while approximation error $\|B-\what B\|_{\F}$ is always decreasing
(or, equivalent, objective $f$ of \eqref{eq:opt-TD} is monotonically increasing).

\subsection{The algorithm}\label{ssec:HOOI}
One advantage of the reformulations in \eqref{eq:obj-TD} is that, fixing $P_{i_1},\ldots,P_{i_{m-1}}$, maximizing objective $f$
over $P_{\ell}\in\STM{k_{\ell}}{n_{\ell}}$ is achieved at $P_{\ell}$ being the top $k_{\ell}$ left singular vector matrix of $C_{\ell}(P_{i_1},\ldots,P_{i_{m-1}})$, i.e., associated with its $k_{\ell}$
largest singular values, or
the top $k_{\ell}$ eigenvector matrix of $H_{\ell}(P_{i_1},\ldots,P_{i_{m-1}})=[C_{\ell}(P_{i_1},\ldots,P_{i_{m-1}})][C_{\ell}(P_{i_1},\ldots,P_{i_{m-1}})]^{\HH}$, i.e., associated with its $k_{\ell}$
largest eigenvalues. This is the key idea of HOOI, as outlined in \Cref{alg:HOOI},
where $C_{\ell}^{(j)}$ and $H_{\ell}^{(j)}$ are computed according to their definitions in  \eqref{eq:CijHij}.
It is noted that if $k_{\ell}=n_{\ell}$ then we can safely ignore optimizing $f$ over $P_{\ell}$ by simply taking $P_{\ell}^{(j)}=I_{n_{\ell}}$ always
        in the iteration.
For the special case: $k_{\ell}=1\,\,\,\forall\ell$,  all $C_{\ell}^{(j)}\in\bbC^{n_{\ell}\times 1}$ are column vectors and consequently
        $P_{\ell}^{(j+1)}=C_{\ell}^{(j)}/\|C_{\ell}^{(j)}\|_2$. For the case, \Cref{alg:HOOI} reduces to the ALS algorithm in
        \cite[section~4.2]{zhgo:2001} and the higher-order power method in \cite[Alg.~3.2]{dldv:2000b}.

\begin{algorithm}[t]
\caption{HOOI: The higher-order orthogonal iteration for TD \cite{dldv:2000b}.}
\label{alg:HOOI}
\begin{algorithmic}[1]
\REQUIRE $m$-mode tensor $B\in\bbC^{n_1\times n_2\times\cdots\times n_m}$,
         integers $1\le k_{\ell}\le n_{\ell}$ for $1\le \ell\le m$,
         and initial $(P_1^{(0)},P_2^{(0)},\ldots,P_m^{(0)})$ with $P_{\ell}\in\STM{k_{\ell}}{n_{\ell}}\,\,\,\forall \ell$;
\ENSURE  an approximate optimal $(P_1,P_2,\ldots,P_m)$ for TD of tensor $B$.
\FOR{$j=0,1,\ldots$ until convergence}
    \FOR{$\ell=1,2\ldots,m$}
        \STATE compute $C_{\ell}^{(j)}$ according to \eqref{eq:CijHij-1};
        \STATE compute $P_{\ell}^{(j+1)}$ as the top $k_{\ell}$ left singular vector matrix of
           $C_{\ell}^{(j)}$, or equivalently
           the top $k_{\ell}$ eigenvector matrix of
           $H_{\ell}^{(j)}$, where  $H_{\ell}^{(j)}$ are as defined in \eqref{eq:CijHij-2};
    \ENDFOR
\ENDFOR
\RETURN the last $(P_1^{(j)},P_2^{(j)},\ldots,P_m^{(j)})$.
\end{algorithmic}
\end{algorithm}

\Cref{alg:HOOI} can be regarded as an extension of the NEPv approach in \cite{li:2024} to solve the system \eqref{eq:KKT-TD} of
$m$ coupled NEPv.
A few comments are in order:
\begin{enumerate}[(i)]
  \item As is in \Cref{alg:HOOI}, each loop on $j$ is just for one $j$.
        However, consecutive $C_{\ell}^{(j)}$ with respective to $\ell$ and then
        to $j$ share common computations that should be taken advantage of for the sake of saving work. For example, for $m=3$,
        it is beneficial
        to do two consecutive $j$ at a time, as the following table suggests:
        $$
        \setlength{\tabcolsep}{4pt}
        \begin{tabular}{|llcl|}
          \hline
        $C_1^{(j)}=C_1(P_2^{(j)},P_3^{(j)})$, & $C_2^{(j)}=C_2(P_1^{(j+1)},P_3^{(j)})$
                   & share  & $B\times_3P_3^{(j)}$, \\
        $C_3^{(j)}=C_3(P_1^{(j+1)},P_2^{(j+1)})$, & $C_1^{(j+1)}=C_1(P_2^{(j+1)},P_3^{(j+1)})$
                   & share  & $B\times_2P_2^{(j+1)}$, \\
        $C_2^{(j+1)}=C_2(P_1^{(j+2)},P_3^{(j+1)})$, & $C_3^{(j+1)}=C_3(P_1^{(j+2)},P_2^{(j+2)})$
                   & share  & $B\times_1P_1^{(j+2)}$. \\
          \hline
        \end{tabular}
        $$
  \item In \Cref{alg:HOOI}, two options are provided to compute $P_{\ell}^{(j+1)}$. The left singular vector option
        is probably the better one when it is computed by a direct method such as by the {\tt svd} function
        in MATLAB, especially when $C_{\ell}^{(j)}$ is a tall skinny matrix, i.e., $k_{i_1}\cdots k_{i_{m-1}}\ll n_{\ell}$. This is what we         do in our current implementation.
        If by some iterative method, there is no need to form $H_{\ell}^{(j)}$ explicitly either but let it live as $[C_{\ell}^{(j)}][C_{\ell}^{(j)}]^{\HH}$
        to perform matrix-vector products.

        It is noted that $\rank(C_{\ell}^{(j)})\le k_{i_1}\cdots k_{i_{m-1}}$ and $P_{\ell}^{(j+1)}$ corresponds to the top $k_{\ell}$ left singular vectors, and hence it would reasonable to impose
        $k_{\ell}\le k_{i_1}\cdots k_{i_{m-1}}$ for each $\ell$.

  \item As stated, $P_{\ell}^{(j+1)}$ is computed to be the top $k_{\ell}$ eigenvector matrix of $H_{\ell}^{(j)}$. This is not necessary
        for simply ensuring that the sequence $\{f_j:=f(P_1^{(j)},P_2^{(j)},\ldots,P_m^{(j)})\}$ of objective value monotonically increases. To guarantee that,
        it suffices to compute $P_{\ell}^{(j+1)}$ so accurate that
        \begin{equation}\label{eq:cond4cvg4HOOI(a)}
        \tr([P_{\ell}^{(j+1)}]^{\HH}H_{\ell}^{(j)}P_{\ell}^{(j+1)})\ge\tr(\big[P_{\ell}^{(j)}\big]^{\HH}H_{\ell}^{(j)}P_{\ell}^{(j)}),
        \end{equation}
        which is sufficient for \Cref{thm:cvg4HOOI}(a) to hold.
        This is good news for a large scale tensor $B$ and $P_{\ell}^{(j+1)}$
        has to be computed iteratively. But to have \Cref{thm:cvg4HOOI}(b,c,d,e), we will need more than
        \eqref{eq:cond4cvg4HOOI(a)}. See \Cref{rk:cvg4HOOI} later.

  \item A natural and cheap stopping criterion is through checking if
        \begin{equation}\label{eq:obj-stop}
        |f_j-f_{j-1}|/f_j\le\epsilon_1,
        \end{equation}
        where $\epsilon_1$ is a preselected tolerance, since the  sequence $\{f_j\}$ of
        objective values monotonically increases (see \Cref{thm:cvg4HOOI} below). However, it has a potential pitfall, namely
        false convergence, if there is a period of iterations where increases in objective value are extremely small.
        To safeguard the false convergence,
        we may check the residual for the KKT condition \eqref{eq:KKT-TD}:
        \begin{equation}\label{eq:KKT-stop}
        \epsilon_{\KKT,j}:=\sum_{\ell=1}^m\frac {\|H_{\ell}(P_{i_1}^{(j)},\ldots,P_{i_{m-1}}^{(j)})P_{\ell}^{(j)}-P_{\ell}^{(j)}\Omega_{\ell}^{(j)}\|_{\F}}
                      {\|B\|_{\F}\|B_{\ufd,\ell}\|_2}\le\epsilon_2,
        \end{equation}
        where $\Omega_{\ell}^{(j)}=\big[P_{\ell}^{(j)}\big]^{\HH}H_{\ell}(P_{i_1}^{(j)},\ldots,P_{i_{m-1}}^{(j)})P_{\ell}^{(j)}$,
        and $\epsilon_2$ is another preselected tolerance. However, according to the way the alternating iteration progresses,
        not all $H_{\ell}(P_{i_1}^{(j)},\ldots,P_{i_{m-1}}^{(j)})$ are computed or available
        because not all
        $C_{\ell}(P_{i_1}^{(j)},\ldots,P_{i_{m-1}}^{(j)})$ are computed. Hence using
        \eqref{eq:KKT-stop} entails nontrivial extra work, and in fact it may be the dominant work for large scale problems as
        we will explain in a moment.
        As a comprise, at Line 1, we  replace \eqref{eq:KKT-stop} with
        \begin{equation}\label{eq:KKT-stop'}
        \tilde\epsilon_{\KKT,j}:=\sum_{\ell=1}^m\frac {\|H_{\ell}^{(j)}P_{\ell}^{(j)}-P_{\ell}^{(j)}\wtd\Omega_{\ell}^{(j)}\|_{\F}}
                 {\|B\|_{\F}\|B_{\ufd,\ell}\|_2}\le\epsilon_2.
        \end{equation}
        where $\wtd\Omega_{\ell}^{(j)}=\big[P_{\ell}^{(j)}\big]^{\HH}H_{\ell}^{(j)}P_{\ell}^{(j)}$
        and $H_{\ell}^{(j)}$ is defined as in \eqref{eq:CijHij}. We may also estimate $\|B_{\ufd,\ell}\|_2$
        by $\sqrt{\|B_{\ufd,\ell}\|_1\|B_{\ufd,\ell}\|_{\infty}}$ for computational efficiency, where
        $\|\cdot\|_p$ for $1\le p\le\infty$ is the matrix $\ell_p$ operator norm
        for $p\in\{1,\infty\}$.
\end{enumerate}

We now present a rough estimate of computational complexity per iterative step, under the assumption
that\footnote {This assumption conceivably makes sense for small $m$ (say, $3$) and small $k_{\ell}$.}
$k_{i_1}\cdots k_{i_{m-1}}\ll n_{\ell}$, i.e.,
each $C_{\ell}(P_{i_1},\ldots,P_{i_{m-1}})$ in \eqref{eq:Ci} is a tall skinny matrix.
Each iterative step of \Cref{alg:HOOI} involves: form $C_{\ell}(P_{i_1},\ldots,P_{i_{m-1}})$ evaluated at most updated $P_{i_1},\ldots,P_{i_{m-1}}$ at the time
and compute its thin SVD, for $1\le \ell\le m$.

Let us start with real tensor $B$, and assume that it is dense.
Forming all $C_{\ell}(P_{i_1},\ldots,P_{i_{m-1}})$, in one inner for-loop: Lines 2--5 of \Cref{alg:HOOI},
costs
\begin{equation}\label{eq:cpx-real-HOOI-Cell}
O\big(n_1n_2\cdots n_m(k_1+k_2+\cdots+k_m)\big)
\end{equation}
flops.
Computing the thin SVD of $C_{\ell}(P_{i_1},\ldots,P_{i_{m-1}})$ is often done in two steps:
a thin QR decomposition followed by the SVD of the $R$-factor from the QR decomposition
at a cost about $6n_{\ell}(k_{i_1}\cdots k_{i_{m-1}})^2+20(k_{i_1}\cdots k_{i_{m-1}})^3$ flops \cite[p.493]{govl:2013} (see also \cite[subsection 3.2]{li:2024}), and hence the cost for all thin SVD, in one inner for-loop: Lines 2--5, is about
$$
O\Big(\sum_{\ell=1}^m n_{\ell}(k_1\cdots k_{\ell-1}k_{\ell+1}\cdots k_m)^2\Big).
$$
Putting it all together, we find the cost per inner for-loop, Lines 2--5 of \Cref{alg:HOOI}, is
\begin{equation}\label{eq:cpx-real-HOOI}
O\Big(n_1n_2\cdots n_m\sum_{\ell=1}^m k_{\ell}+\sum_{\ell=1}^m n_{\ell}(k_1\cdots k_{\ell-1}k_{\ell+1}\cdots k_m)^2\Big).
\end{equation}
flops. It can be seen that the dominant part is from forming $C_{\ell}(P_{i_1},\ldots,P_{i_{m-1}})$.

If $B$ is complex, then the complex arithmetic will be used throughout. Each complex addition/subtraction takes two flops while
multiplication takes 6. Roughly speaking, for complex $B$, the  computational complexity per iterative step is about 4 times as much as
the estimate in \eqref{eq:cpx-real-HOOI} (assuming there are about an equal number of addition/subtraction and multiplication operations).

\begin{remark}\label{rk:Xu2018-align}
$P_{\ell}^{(j+1)}$ calculated at Lines 2-5 for $1\le\ell\le m$ are inherently non-unique. But this does not matter so long
as $\cR(P_{\ell}^{(j+1)})$ is an accurate approximation to the top eigenspace of $H_{\ell}^{(j)}$
because $f$ is componentwise right-unitarily invariant in the sense of
\eqref{eq:f-UI} and so are $H_{\ell}$ for $1\le\ell\le m$ in \eqref{eq:Hi} in the sense of \eqref{eq:Hi-UI}.
To remove such an inherent non-uniqueness, Xu~\cite{xu:2018} proposed to align $P_{\ell}^{(j+1)}$ to $P_{\ell}^{(j)}$
each time it is computed, e.g., replacing $P_{\ell}^{(j+1)}$ with $P_{\ell}^{(j+1)}Q_{\ell}^{(j)}$ where $Q_{\ell}^{(j)}$ is
an orthonormal polar factor of $\big[P_{\ell}^{(j+1)}\big]^{\HH}P_{\ell}^{(j)}$, according to \Cref{lm:basis-align}.
It is known that $Q_{\ell}^{(j)}$ is unique if $\rank(\big[P_{\ell}^{(j+1)}\big]^{\HH}P_{\ell}^{(j)})=k$
\cite{li:1993b,li:1995}. This rank condition is eventually satisfied if $\cR(P_{\ell}^{(j+1)})$ is convergent.
With that being said, we point out that such an alignment is numerically unnecessary as far as the quality of
the eventual approximate TD is concerned. In our convergence theorem, \Cref{thm:cvg4HOOI}, below,
we indeed include a result which says, under a condition, $P_{\ell}^{(j+1)}$ can be aligned {\em in theory\/}
to yield
convergent sequence, but the alignments, such as in \cite{xu:2018}, need not to happen numerically.
\end{remark}

\subsection{Convergence analysis}\label{ssec:cvg-HOOI}
Next we will perform a convergence analysis of HOOI.
In what follows, we stick to the assignments
to $C_{\ell}^{(j)}$, $H_{\ell}^{(j)}$, and $\Lambda_{\ell}^{(j)}$  in \eqref{eq:CijHij},
and  let
\begin{equation}\label{eq:Lambdaij}
\delta_{\ell}^{(j)}=\lambda_{k_{\ell}}(H_{\ell}^{(j)})-\lambda_{k_{\ell}+1}(H_{\ell}^{(j)})
\quad\mbox{for $1\le \ell\le m$}.
\end{equation}

\begin{theorem}\label{thm:cvg4HOOI}
Let the sequence $\{(P_1^{(j)},P_2^{(j)},\ldots,P_m^{(j)})\}_{j=0}^{\infty}$ be generated by \Cref{alg:HOOI},
$(P_{*1},P_{*2},\ldots,P_{*m})$ be an accumulation point of $\{(P_1^{(j)},P_2^{(j)},\ldots,P_m^{(j)})\}_{j=0}^{\infty}$ whose
subsequence $\{(P_1^{(j)},P_2^{(j)},\ldots,P_m^{(j)})\}_{j\in\bbI}$ converges to $(P_{*1},P_{*2},\ldots,P_{*m})$, and
let $(\what P_{*1},\what P_{*2},\ldots,\what P_{*\,m-1})$ be an accumulation point of $\{(P_1^{(j+1)},P_2^{(j+1)},\ldots,P_{m-1}^{(j+1)})\}_{j\in\bbI}$.
Denote by\footnote{Notice that $\what H_{*1}=H_1(P_{*2},\ldots,P_{*m})=H_{*1}$, and, by convention,
      $H_{*m}=H_m(P_{*1},\ldots,P_{*\,m-1})$.}, for $1\le\ell\le m$,
\begin{subequations}\label{eq:Hell-star-all}
\begin{align}
H_{*\ell}&:=H_{\ell}(P_{*1},\ldots,P_{*\,\ell-1},P_{*\,\ell+1},\ldots,P_{*m}), \label{eq:Hell-star} \\
\what H_{*\ell}&:=H_{\ell}(\what P_{*1},\ldots,\what P_{*\,\ell-1},P_{*\,\ell+1},\ldots,P_{*m}). \label{eq:hatHell-star}
\end{align}
\end{subequations}
The following statements hold.
\begin{enumerate}[{\rm (a)}]
  \item The sequence $\{f(P_1^{(j)},P_2^{(j)},\ldots,P_m^{(j)})\}_{j=0}^{\infty}$ is monotonically increasing and convergent;
  \item We have, for $1\le\ell\le m$,
        \begin{equation}\label{eq:TD:KKTatMAX'-HOOI}
        \what H_{*\ell}\,P_{*\ell}=P_{*\ell}\Lambda_{*\ell},
        \end{equation}
        where
        $\Lambda_{*\ell}=P_{*\ell}^{\HH}\what H_{*\ell}\,P_{*\ell}\in\bbC^{k_{\ell}\times k_{\ell}}$ is Hermitian whose eigenvalues are the $k_{\ell}$ largest eigenvalues
        of $\what H_{*\ell}$  for $1\le \ell\le m$, respectively.
  \item In item~{\rm (b)},
        if \footnote {A necessary condition for having \eqref{eq:pos-gap2} and \eqref{eq:pos-gap3} below is
                     $k_{\ell}\le k_{i_1}\ldots k_{i_{m-1}}=k_1\cdots k_{\ell-1}k_{\ell+1}\cdots k_m$ for $1\le \ell\le m$.},
        for $1\le\ell\le m-1$,
        \begin{equation}\label{eq:pos-gap2}
        \delta_{*\ell}:=\lambda_{k_{\ell}}(H_{*\ell})-\lambda_{k_{\ell}+1}(H_{*\ell})>0,
        \end{equation}
        then $\what P_{*\ell}=P_{*\ell}Q_{*\ell}$ for some $Q_{*\ell}\in\STM{k_{\ell}}{n_{\ell}}$ for $1\le\ell\le m-1$
        and, because of \eqref{eq:Hi-UI}, \eqref{eq:TD:KKTatMAX'-HOOI} becomes, for $1\le\ell\le m$,
        \begin{equation}\label{eq:TD:KKTatMAX'-HOOI'}
        H_{*\ell}\,P_{*\ell}=P_{*\ell}\Lambda_{*\ell},
        \end{equation}
        where $\Lambda_{*\ell}=P_{*\ell}^{\HH} H_{*\ell}\,P_{*\ell}\in\bbC^{k_{\ell}\times k_{\ell}}$ is Hermitian whose eigenvalues are the $k_{\ell}$ largest eigenvalues
        of $H_{*\ell}$  for $1\le \ell\le m$, respectively.
  \item Define, for the purpose of alignment in theory,
        $Q_{\ell}^{(j)}\in\STM{k_{\ell}}{k_{\ell}}$ to be the orthonormal polar
        factor of $\big[P_{\ell}^{(j)}\big]^{\HH}P_{*\ell}$ for $1\le \ell\le m$ and $j\ge 0$.
        If $(P_{*1},P_{*2},\ldots,P_{*m})$ is an isolated accumulation point of $\{(P_1^{(j)}Q_1^{(j)},P_2^{(j)}Q_2^{(j)},\ldots,P_m^{(j)}Q_m^{(j)})\}_{j=0}^{\infty}$
        and if, besides \eqref{eq:pos-gap2} for $1\le \ell\le m-1$, also
        \begin{equation}\label{eq:pos-gap3}
        \delta_{*m}:=\lambda_{k_{m}}(H_{*m})-\lambda_{k_{m}+1}(H_{*m})>0,
        \end{equation}
        then the entire sequence $\{(P_1^{(j)}Q_1^{(j)},P_2^{(j)}Q_2^{(j)},\ldots,P_m^{(j)}Q_m^{(j)})\}_{j=0}^{\infty}$ converges to $(P_{*1},P_{*2},\ldots,P_{*m})$.

  \item We have $2m$ convergent series, for $1\le \ell\le m$,
        \begin{subequations}\label{eq:cvg4HOOI:TD:series}
        \begin{align}
        \sum_{j=0}^{\infty}\delta_{\ell}^{(j)}\,
                         \big\|\sin\Theta\big(\cR(P_{\ell}^{(j+1)}),\cR(P_{\ell}^{(j)})\big)\big\|_{\F}^2
                      &<\infty,    \label{eq:cvg4HOOI:TD:series-1} \\
        \sum_{j=0}^{\infty}\delta_{\ell}^{(j)}\,
                  \frac {\big\|H_{\ell}^{(j)}P_{\ell}^{(j)}-P_{\ell}^{(j)}\Lambda_{\ell}^{(j)}\big\|_{\F}^2}
                        {\big\|H_{\ell}^{(j)}\big\|_{\F}^2}
                      &<\infty.                 \label{eq:cvg4HOOI:TD:series-2}
        \end{align}
        \end{subequations}
\end{enumerate}
\end{theorem}

\begin{proof}
Recall \eqref{eq:CijHij} and \eqref{eq:flow-GS}, and let, for $1\le\ell\le m$,
\begin{subequations}\label{eq:cvg4HOOI:TD:pf-1}
\begin{align}
\eta_{j+\ell/m}&=\tr\big(\big[P_{\ell}^{(j+1)}\big]^{\HH}H_{\ell}^{(j)}\big[P_{\ell}^{(j+1)}\big]\big)-\tr\big(\big[P_{\ell}^{(j)}\big]^{\HH}H_{\ell}^{(j)}\big[P_{\ell}^{(j)}\big]\big)
            \label{eq:cvg4HOOI:TD:pf-1aa}\\
            &=f(P_1^{(j+1)},\ldots,P_{\ell}^{(j+1)},P_{\ell+1}^{(j)},\ldots,P_m^{(j)})
              -f(P_1^{(j+1)},\ldots,P_{\ell-1}^{(j+1)},P_{\ell}^{(j)},\ldots,P_m^{(j)}). \label{eq:cvg4HOOI:TD:pf-1ab}
\end{align}
\end{subequations}
All $\eta_{j+\ell/m}$ for $1\le\ell\le m$ are non-negative because each $P_{\ell}^{(j+1)}$ is the top $k_{\ell}$
eigenvector matrix of $H_{\ell}^{(j)}$,
yielding
\begin{align*}
f(P_1^{(j+1)},P_2^{(j+1)},\ldots,P_m^{(j+1)})&\ge f(P_1^{(j+1)},\ldots,P_{m-1}^{(j+1)},P_m^{(j)})\\
    &\ge \cdots \\
    &\ge f(P_1^{(j)},P_2^{(j)},\ldots,P_m^{(j)}).
\end{align*}
This proves item (a).
Along the way, we also showed
\begin{align}
f(P_1^{(j+1)},&P_2^{(j+1)},\ldots,P_m^{(j+1)})= f(P_1^{(j)},P_2^{(j)},\ldots,P_m^{(j)})\nonumber \\
   &+\sum_{\ell=1}^m\Big[\tr\big(\big[P_{\ell}^{(j+1)}\big]^{\HH}H_{\ell}^{(j)}\big[P_{\ell}^{(j+1)}\big]\big)
          -\tr\big(\big[P_{\ell}^{(j)}\big]^{\HH}H_{\ell}^{(j)}\big[P_{\ell}^{(j)}\big]\big)\Big].
        \label{eq:cvg4HOOI:TD:pf-3}
\end{align}

We now prove item~(b).
Without loss of generality, we may  assume that also
$$
\{(P_1^{(j+1)},P_2^{(j+1)},\ldots,,P_{m-1}^{(j+1)})\}_{j\in\bbI}
$$
converges to
$(\what P_{*1},\what P_{*2},\ldots,\what P_{*\,m-1})$; otherwise,
we can pick up a convergent subsequence of $\{(P_1^{(j+1)},P_2^{(j+1)},\ldots,,P_{m-1}^{(j+1)})\}_{j\in\bbI}$
and reassign $\bbI$ according to the convergent subsequence.
We have
\begin{subequations}\label{eq:cvg4HOOI:TD:pf-5}
\begin{alignat}{2}
\lim_{\bbI\ni j\to \infty}P_{\ell}^{(j)}&=P_{*\ell} &\quad&\mbox{for $1\le\ell\le m$}, \\
\lim_{\bbI\ni j\to \infty}P_{\ell}^{(j+1)}&=\what P_{*\ell} &\quad&\mbox{for $1\le\ell\le m-1$}.
\end{alignat}
\end{subequations}
Denote by $\tr_{\max,k}(\cdot)$ the sum of the $k$ largest eigenvalues of a Hermitian matrix, and let
$$
\delta_{\ell}:=\tr_{\max,k_{\ell}}(\what H_{*\ell})
                   -\tr(P_{*\ell}^{\HH}\what H_{*\ell}P_{*\ell})\ge 0
                   \quad\mbox{for $1\le\ell\le m$}.
$$
Assume, to the contrary, that at least one of the equations in \eqref{eq:TD:KKTatMAX'-HOOI}
with $\Lambda_{*\ell}$ as described there did not hold. Then $\delta=\sum_{\ell=1}^m\delta_{\ell}>0$.
Since $\tr_{\max,k_{\ell}}(H_{\ell}(P_{i_1},\ldots,P_{i_{m-1}}))$, $\tr(P_{\ell}^{\HH}H_{\ell}(P_{i_1},\ldots,P_{i_{m-1}})P_{\ell})$ for $1\le \ell\le m$ and $f(P_1,P_2,\ldots,P_m)$
are continuous in $(P_1,P_2,\ldots,P_m)$ with $P_{\ell}\in\bbC^{n_{\ell}\times k_{\ell}}\,\,\,\forall\ell$,
there is an $j_0\in\bbI$ such that, for $1\le\ell\le m$,
\begin{subequations}\label{eq:cvg4HOOI:TD:pf-7}
\begin{gather}
\Big|\tr_{\max,k_{\ell}}(\what H_{*\ell})-\tr_{\max,k_{\ell}}(H_{\ell}^{(j_0)})\Big| <\frac {\delta}{2m+1}, \label{eq:thm:cvg4TDm:pf-7a}\\
\Big|\tr(P_{*\ell}^{\HH}\what H_{*\ell}P_{*\ell})
               -\tr([P_{\ell}^{(j_0)}]^{\HH}H_{\ell}^{(j_0)}P_{\ell}^{(j_0)})\Big|
             <\frac {\delta}{2m+1}, \label{eq:thm:cvg4TDm:pf-7b} \\
f_*-\frac {\delta}{2(2m+1)}<f(P_1^{(j_0)},P_2^{(j_0)},\ldots,P_m^{(j_0)})\le f_*,
       \label{eq:thm:cvg4TDm-subs-1:pf-7e}
\end{gather}
\end{subequations}
where $f_*:=f(P_{*1},P_{*2},\ldots,P_{*m})$.
By \eqref{eq:cvg4HOOI:TD:pf-3} and noticing
$\tr\big(\big[P_{\ell}^{(j+1)}\big]^{\HH}H_{\ell}^{(j)}\big[P_{\ell}^{(j+1)}\big]\big)=\tr_{\max,k_{\ell}}(H_{\ell}^{(j)})$,
we have
\begin{align}
f(P_1^{(j_0+1)},&P_2^{(j_0+1)},\ldots,P_m^{(j_0+1)})
   >f_*-\frac {\delta}{2(2m+1)} \nonumber \\
  & +\sum_{\ell=1}^m\Big[\tr_{\max,k_{\ell}}(\what H_{*\ell})-\frac {\delta}{2m+1}
               -\tr(P_{*\ell}^{\HH}\what H_{*\ell}P_{*\ell})-\frac {\delta}{2m+1}\Big] \nonumber\\
  &=f_*+\frac {\delta}{2(2m+1)}>f_*, \label{eq:thm:cvg4TDm-subs-1:pf-8}
\end{align}
contradicting $f(P_1^{(i)},P_2^{(i)},\ldots,P_m^{(i)})\le\lim_{j\to\infty}f(P_1^{(j)},P_2^{(j)},\ldots,P_m^{(j)})= f_*$ for all $i$.
This completes the proof of item~(b).

We note that \eqref{eq:TD:KKTatMAX'-HOOI'} for $\ell=1$ is the same as \eqref{eq:TD:KKTatMAX'-HOOI} for $\ell=1$.
In what follows, we will show, by induction, that
there exists a $Q_{*\ell}\in\STM{k_{\ell}}{k_{\ell}}$ such that $\what P_{*\ell}=P_{*\ell}Q_{*\ell}$ for $1\le\ell\le m-1$.
Since we always have
\begin{equation}\label{eq:H1-part-eigD}
H_1(P_2^{(j)},\ldots,P_m^{(j)})\,P_1^{(j+1)}=P_1^{(j+1)}\Lambda_1^{(j)},
\end{equation}
a partial eigen-decomposition such that
$\eig(\Lambda_1^{(j)})$ consists of the $k_1$ largest eigenvalues of $H_1(P_2^{(j)},\ldots,P_m^{(j)})$. Letting $\bbI\ni j\to\infty$
yields $H_{*1}\,\what P_{*1}=\what P_{*1}\what \Lambda_{*1}$, a partial eigen-decomposition also such that
$\eig(\what\Lambda_{*1})$ consists of the $k_1$ largest eigenvalues of $H_{*1}$.
If \eqref{eq:pos-gap2} for $\ell=1$ holds, then the eigenspace associated with the $k_1$ largest eigenvalues of $H_{*1}$
is unique and thus there exists a $Q_{*1}\in\STM{k_1}{k_1}$ such that $\what P_{*1}=P_{*1}Q_{*1}$.
Suppose now that there exist  $Q_{*\ell}\in\STM{k_{\ell}}{k_{\ell}}$ such that $\what P_{*\ell}=P_{*\ell}Q_{*\ell}$
for $1\le\ell\le s< m-1$.
Since we have
$$
H_{s+1}(P_1^{(j+1)},\ldots,P_s^{(j+1)},P_{s+2}^{(j)},\ldots,P_m^{(j)})\,P_{s+1}^{(j+1)}=P_{s+1}^{(j+1)}\Lambda_{s+1}^{(j)},
$$
a partial eigen-decomposition such that
$\eig(\Lambda_{s+1}^{(j)})$ consists of the $k_{s+1}$ largest eigenvalues of
$H_{s+1}(P_1^{(j+1)},\ldots,P_s^{(j+1)},P_{s+2}^{(j)},\ldots,P_m^{(j)})$.
Letting $\bbI\ni j\to\infty$
yields
$$
H_{s+1}(\what P_{*1},\ldots,\what P_{s*},P_{*\,s+2},\ldots,P_{*m})\what P_{*\,s+1}=\what P_{*\,s+1}\what \Lambda_{*\,s+1},
$$
i.e., $H_{*\,s+1}\,\what P_{*\,s+1}=\what P_{*\,s+1}\what \Lambda_{*\,s+1}$ by the induction hypothesis that
$\what P_{*\ell}=P_{*\ell}Q_{\ell}$ for $1\le\ell\le s< m-1$. It is
a partial eigen-decomposition also such that
$\eig(\what\Lambda_{*\,s+1})$ consists of the $k_{s+1}$ largest eigenvalues of $H_{*\,s+1}$  as well.
This, together with \eqref{eq:pos-gap2} for $\ell=s+1$, imply that
there exists  $Q_{*\,s+1}\in\STM{k_{s+1}}{k_{s+1}}$ such that $\what P_{*\,s+1}=P_{*\,s+1}Q_{*\,s+1}$, completing the induction proof,
and, at the same time, the proof of item~(c).

We now prove item~(d) with the help of \Cref{lm:isolatedconvg} to be stated later. To simplify notation, we will write
$\wtd P_{\ell}^{(j)}=P_{\ell}^{(j)}Q_{\ell}^{(j)}$.
Let $\{(\wtd P_1^{(j)},\wtd P_2^{(j)},\ldots,\wtd P_m^{(j)})\}_{j\in\bbI_1}$ be any convergent subsequence that converges to $(P_{*1},P_{*2},\ldots,P_{*m})$.
Since $\{(\wtd P_1^{(j+1)},\wtd P_2^{(j+1)},\ldots,\wtd P_m^{(j+1)})\}_{j\in\bbI_1}$ is a bounded sequence, it has a convergent subsequence
$\{(\wtd P_1^{(j+1)},\wtd P_2^{(j+1)},\ldots,\wtd P_m^{(j+1)})\}_{j\in\bbI_2}$ that converges to
$(\what P_{*1},\what P_{*2},\ldots,\what P_{*m})$,
where $\bbI_2\subseteq\bbI_1$.
As before, we have \eqref{eq:H1-part-eigD},
yielding
\begin{align*}
H_1(\wtd P_2^{(j)},\ldots,\wtd P_m^{(j)})\wtd P_1^{(j+1)}&=H_1(P_2^{(j)},\ldots,P_m^{(j)})P_1^{(j+1)}Q_1^{(j+1)} \\
  &=\wtd P_1^{(j+1)}\underbrace{(\big[Q_1^{(j+1)}\big]^{\HH}\Lambda_1^{(j)}Q_1^{(j+1)})}_{=:\wtd\Lambda_1^{(j)}}.
\end{align*}
Letting $\bbI_1\supseteq\bbI_2\ni j\to\infty$ yields
$H_{*1}\what P_{*1}=\what P_{*1}\wtd\Lambda_{*1}$, which as we argued before implies
$\what P_{*1}=P_{*1}Q_{1}$ for some $Q_{1}\in\STM{k_1}{k_1}$. We claim that $\what P_{*1}=P_{*1}$, i.e., $Q_{1}=I_{k_1}$. To this end,
we denote by $\Theta_1^{(j+1)}:=\Theta(\cR(\wtd P_1^{(j+1)}),\cR(P_{*1}))$ and
notice that $\cR(\what P_{*1})=\cR(P_{*1})$ and by \eqref{eq:basis-align}
\begin{equation}\label{eq:thm:cvg4TDm-subs-1:pf-9}
\|\wtd P_1^{(j+1)}-P_{*1}\|_{\F}=2\|\sin(\Theta_1^{(j+1)}/2)\|_{\F}.
\end{equation}
Because
$$
\lim_{\bbI_2\ni j\to\infty}\wtd P_1^{(j+1)}=\what P_{*1}
\quad\mbox{implies}\quad
\lim_{\bbI_2\ni j\to\infty}\Theta_1^{(j+1)}=\Theta(\cR(\what P_{*1}),\cR(P_{*1}))=0,
$$
we get, by \eqref{eq:thm:cvg4TDm-subs-1:pf-9}, $\|\wtd P_1^{(j+1)}-P_{*1}\|_{\F}\to 0$ as $\bbI_1\supseteq\bbI_2\ni j\to\infty$. Also,
as $\bbI_1\supseteq\bbI_2\ni j\to\infty$, $\|\wtd P_1^{(j+1)}-P_{*1}\|_{\F}\to\|\what P_{*1}-P_{*1}\|_{\F}$.
We conclude $\|\what P_{*1}-P_{*1}\|_{\F}=0$, leading to
$\what P_{*1}=P_{*1}$. Similarly, we can prove $\what P_{*\ell}=P_{*\ell}$ for $\ell=2,\ldots,m$, too.
Hence as $\bbI_2\ni j\to\infty$
$$
\|\wtd P_1^{(j)}-\wtd P_1^{(j+1)}\|_{\F}\le\|\wtd P_1^{(j)}- P_{*1}\|_{\F}+\| P_{*1}-\wtd P_1^{(j+1)}\|_{\F}\to 0,
$$
and similarly $\|\wtd P_{\ell}^{(j)}-\wtd P_{\ell}^{(j+1)}\|_{\F}\to 0$ for $\ell=2,\ldots,m$.
Now use Lemma~\ref{lm:isolatedconvg} below to conclude that the entire sequence
$\{(\wtd P_1^{(j)},\wtd P_2^{(j)},\ldots,\wtd P_m^{(j)})\}_{j=0}^{\infty}$ converges to $(P_{*1},P_{*2},\ldots,P_{*m})$, as needed.

Finally for item (e), we return to \eqref{eq:cvg4HOOI:TD:pf-3}. By \Cref{lm:maxtrace-proj},
we have, for $1\le \ell\le m$,
\begin{equation}\label{eq:thm:cvg4TDm-subs-1:pf-10}
\frac {\big\|H_{\ell}^{(j)}P_{\ell}^{(j)}-P_{\ell}^{(j)}\Lambda_{\ell}^{(j)}\big\|_{\F}^2}
                        {[\lambda_1(H_{\ell}^{(j)})-\lambda_{n_{\ell}}(H_{\ell}^{(j)})]^2}
    \le\big\|\sin\Theta\big(\cR(P_{\ell}^{(j+1)}),\cR(P_{\ell}^{(j)})\big)\big\|_{\F}^2
    \le\frac {\eta_{j+\ell/m}}{\delta_{\ell}^{(j)}}.
\end{equation}
Hence \eqref{eq:cvg4HOOI:TD:series-2} is a consequence of \eqref{eq:cvg4HOOI:TD:series-1} together with
$$
0\le\lambda_1(H_{\ell}^{(j)})-\lambda_{n_{\ell}}(H_{\ell}^{(j)})\le 2\|H_{\ell}^{(j)}\|_2\le 2\|H_{\ell}^{(j)}\|_{\F}.
$$
To see \eqref{eq:cvg4HOOI:TD:series-1}, we combine
$$
2\sum_{j=0}^{\infty}\sum_{\ell=1}^m\eta_{j+\ell/m}
     \le\lim_{j\to\infty} f(P_1^{(j+1)},\ldots,P_m^{(j+1)})-f(P_1^{(0)},\ldots,P_m^{(0)})
     <\infty,
$$
with the second inequality in \eqref{eq:thm:cvg4TDm-subs-1:pf-10}.
The proof is completed.
\end{proof}

The following lemma was used in the proof above and it is an equivalent restatement of \cite[Lemma 4.10]{moso:1983}
(see also \cite[Proposition 7]{kaqi:1999}) in the context of a metric space.

\begin{lemma}[{\cite[Lemma 4.10]{moso:1983}}]\label{lm:isolatedconvg}
Let $\scrG$ be a metric space with metric $\dist(\cdot,\cdot)$, and let
$\{\by_i\}_{i=0}^{\infty}$ be a sequence in $\scrG$. If
$\by_*\in \scrG$ is an isolated accumulation point 	
of the sequence such that, for every subsequence $\{\by_i\}_{i\in\bbI}$
converging to $\by_*$, there is an infinite subset $\widehat{\bbI}\subseteq \bbI$ satisfying
$\dist(\by_i,\by_{i+1})\to 0$ as $\what\bbI\ni i\to\infty$,
then the entire sequence $\{\by_i\}_{i=0}^{\infty}$ converges to $\by_*$.
\end{lemma}

\begin{remark}\label{rk:cvg4HOOI}
\Cref{thm:cvg4HOOI} is proved under the condition that, at Lines 2 -- 5 of \Cref{alg:HOOI},
$P_{\ell}^{(j+1)}$ as the top $k_{\ell}$ left singular vector matrix of
$C_{\ell}^{(j)}$, or equivalently the top $k_{\ell}$ eigenvector matrix of $H_{\ell}^{(j)}$. This can be too demanding,
especially when $n_{\ell}$ or $k_{i_1}\ldots k_{i_{m-1}}$ is very large and $P_{\ell}^{(j+1)}$ has to be computed iteratively.
Immediately after stating \Cref{alg:HOOI}, we commented that for \Cref{thm:cvg4HOOI}(a) it is sufficient
to compute $P_{\ell}^{(j+1)}$ so that \eqref{eq:cond4cvg4HOOI(a)} holds, as one can see from the proof above.
One the other hand, we will need more than \eqref{eq:cond4cvg4HOOI(a)} to guarantee
\Cref{thm:cvg4HOOI}(b -- e), namely, we will need
\begin{equation}\label{eq:cond4cvg4HOOI(b--e)}
\tr([P_{\ell}^{(j+1)}]^{\HH}H_{\ell}^{(j)}P_{\ell}^{(j+1)})-\tr(\big[P_{\ell}^{(j)}\big]^{\HH}H_{\ell}^{(j)}P_{\ell}^{(j)})
   \ge c\big[\tr_{\max,k_{\ell}}(H_{\ell}^{(j)})-\tr(\big[P_{\ell}^{(j)}\big]^{\HH}H_{\ell}^{(j)}P_{\ell}^{(j)})\big],
\end{equation}
where $0<c\le 1$ is a constant, independent of $\ell$ and $j$. Carefully examining the proof above,
\eqref{eq:cond4cvg4HOOI(b--e)} will allow us to have something like
\eqref{eq:thm:cvg4TDm-subs-1:pf-8}, i.e.,
$f(P_1^{(j_0+1)},P_2^{(j_0+1)},\ldots,P_m^{(j_0+1)})>f_*$.
\end{remark}

As a corollary of Theorem~\ref{thm:cvg4HOOI}(e),
if $\delta_{\ell}^{(j)}=\lambda_{k_{\ell}}(H_{\ell}^{(j)})-\lambda_{k_{\ell}+1}(H_{\ell}^{(j)})$ is eventually bounded below away from $0$
uniformly\footnote {By which we mean that there exist a constant $c>0$ and an integer $K$ such that
  $\delta_{\ell}^{(j)}\ge c$ for all $j\ge K$.},
then
\begin{equation}\label{eq:NEPv-always}
\lim_{j\to\infty}\frac {\big\|H_{\ell}^{(j)}P_{\ell}^{(j)}-P_{\ell}^{(j)}\Lambda_{\ell}^{(j)}\big\|_{\F}}
                        {\big\|H_{\ell}^{(j)}\big\|_{\F}} =0,
\end{equation}
namely, increasingly $H_{\ell}^{(j)}P_{\ell}^{(j)}\approx P_{\ell}^{(j)}\Lambda_{\ell}^{(j)}$
towards a partial eigen-decomposition of $H_{\ell}^{(j)}$ with $\eig(\Lambda_{\ell}^{(j)})$ being composed of the top $k_{\ell}$
eigenvalues of $H_{\ell}^{(j)}$. This
means that $(P_1^{(j)},P_2^{(j)},\ldots,P_m^{(j)})$ becomes a more and more accurate approximate solution
to the KKT condition \eqref{eq:KKT-TD},
even in the absence of knowing whether the entire sequence $\{(P_1^{(j)},P_2^{(j)},\ldots,P_m^{(j)})\}_{j=0}^{\infty}$ converges or not.

Implication in \eqref{eq:NEPv-always} by \Cref{thm:cvg4HOOI}(e) is perhaps more significant
than the conclusions on an accumulation point $(P_{*1},P_{*2},\ldots,P_{*m})$ in \Cref{thm:cvg4HOOI}(b,c) due to our previously
comments that maximizer tuples  of \eqref{eq:opt-TD} are inherently non-unique. Rather than claiming
entrywise convergence of computed approximations, the implication shows that
computed maximizer tuples progress to satisfy the KKT condition while, at the same time,
the objective moves up as guaranteed by \Cref{thm:cvg4HOOI}(a).

\begin{remark}\label{rk:Xu2018-cvg}
Previously, Xu~\cite{xu:2018} proposed to align $P_{\ell}^{(j+1)}$ and $P_{\ell}^{(j)}$ in a way like \Cref{lm:basis-align} suggests,
every time $P_{\ell}^{(j+1)}$ is computed at Line 4 of \Cref{alg:HOOI}, yielding the so-called  greedy HOOI, and then
performed convergence analysis of it.
Numerically as far as computing an approximate TD of the same quality is concerned, this alignment is unnecessary as we explained in
\cref{sec:TD}. Xu's analysis is for the greedy HOOI on real tensors only, and  it
relies on the result of \cite{bost:2014} which is based on the Kurdyka–Lojasiewicz theory for semi-algebraic functions that the numerical linear algebra community is not familiar to. Our analysis, however, is more accessible by the community as it
is entirely based on numerical linear algebra knowledge, applies to both real and complex tensors,
and on HOOI itself.
\end{remark}


\section{Alternating Subspace Iteration (ASI)}\label{sec:ASI}
\subsection{The algorithm}\label{ssec:ASI}
Previously in HOOI, $P_{\ell}^{(j+1)}$ is computed to {\em exactly\/} minimize $\tr(P_{\ell}^{\HH}H_{\ell}^{(j)}P_{\ell})$ over $P_{\ell}\in\STM{k_{\ell}}{n_{\ell}}$.
Alternatively, one may use one-step subspace iteration on $H_{\ell}^{(j)}$
\cite{demm:1997,govl:2013,stew:2001a} to improve $P_{\ell}^{(j)}$. The basic idea is to let $P_{\ell}^{(j+1)}\in\STM{k_{\ell}}{n_{\ell}}$
be some orthonormal basis matrix
of $\cR(H_{\ell}^{(j)}P_{\ell})$.
Kroonenberg and De Leeuw~\cite{krdl:1980} appeared to be the first researchers to do so. They also proposed to use essentially the truncated
HOSVD for their initial guess, whereas Tucker~\cite{tuck:1966} simply used  the initial guess as their final solution for TD,
which is not optimal in the sense of minimizing $\|B-\what B\|_{\F}$, and, in fact, still far from some stationary point.
It is interesting to note that  the term HOSVD first explicitly appeared  in \cite{dldv:2000a} in 2000.

There are a number of ways to choose $P_{\ell}^{(j+1)}$ from one-step subspace iteration:
\begin{enumerate}[(1)]
  \item the thin QR decomposition: $H_{\ell}^{(j)}P_{\ell}^{(j)}=P_{\ell}^{(j+1)} R_{\ell}^{(j)}$ where $R_{\ell}^{(j)}$ is upper triangular;
  \item the polar decomposition: $H_{\ell}^{(j)}P_{\ell}^{(j)}=P_{\ell}^{(j+1)}\Lambda_{\ell}^{(j)}$ where $\Lambda_{\ell}^{(j)}\succeq 0$.
\end{enumerate}
Both are numerically stable, and even when $\rank(H_{\ell}^{(j)}P_{\ell}^{(j)})<k_{\ell}$, they can still produce numerically accurate $P_{\ell}^{(j+1)}\in\STM{k_{\ell}}{n_{\ell}}$.
Computed $P_{\ell}^{(j+1)}$ by the two methods are in general different, but it does not pose any difference in approximation accuracy
as far as the objective value is concerned. The thin QR decomposition is cheaper. When
$\rank(H_{\ell}^{(j)}P_{\ell}^{(j)})=k_{\ell}$, $P_{\ell}^{(j+1)}$ by the polar decomposition is unique and it by the thin QR decomposition can also be
uniquely specified, too, by enforcing the diagonal entries of $R_{\ell}^{(j)}$ positive.
\Cref{alg:ASI} outlines the resulting algorithm -- {\em alternating subspace iteration\/} (ASI) for TD. For convenience of our
convergence analysis, we state the algorithm with the polar decomposition; otherwise we will have to introduce
basis alignment in theory as we did in \Cref{thm:cvg4HOOI} for HOOI in \Cref{alg:HOOI}.
\Cref{alg:ASI} can be regarded as an extension of the NPDo approach in \cite{li:2024} to solve $m$ coupled NPDo
in \eqref{eq:KKT-TD}.

\begin{algorithm}[t]
\caption{ASI: the alternating subspace iteration for TD \cite{krdl:1980}.}
\label{alg:ASI}
\begin{algorithmic}[1]
\REQUIRE $m$-mode tensor $B\in\bbC^{n_1\times n_2\times\cdots\times n_m}$,
         integers $1\le k_{\ell}\le n_{\ell}$ for $1\le \ell\le m$,
         and initial $(P_1^{(0)},P_2^{(0)},\ldots,P_m^{(0)})$ with $P_{\ell}\in\STM{k_{\ell}}{n_{\ell}}\,\,\,\forall \ell$;
\ENSURE  an approximate optimal TD of tensor $B$.
\FOR{$j=0,1,\ldots$ until convergence}
    \FOR{$\ell=1,2\ldots,m$}
        \STATE compute $H_{\ell}^{(j)}P_{\ell}^{(j)}$ according to $[C_{\ell}^{(j)}]([C_{\ell}^{(j)}]^{\HH}P_{\ell}^{(j)})$,
               where $C_{\ell}^{(j)}$ and $H_{\ell}^{(j)}$ are as defined in \eqref{eq:CijHij};
        \STATE compute $P_{\ell}^{(j+1)}$  as the orthonormal polar factor of
               $H_{\ell}^{(j)}P_{\ell}^{(j)}$;
    \ENDFOR
\ENDFOR
\RETURN the last $(P_1^{(j)},P_2^{(j)},\ldots,P_m^{(j)})$.
\end{algorithmic}
\end{algorithm}

In the case of the polar decomposition,
if $\rank(H_{\ell}^{(j)}P_{\ell}^{(j)})=k_{\ell}$, then also
\begin{equation}\label{eq:krdl1980}
P_{\ell}^{(j+1)}=H_{\ell}^{(j)}P_{\ell}^{(j)}\big(\big[P_{\ell}^{(j)}\big]^{\HH}\big[H_{\ell}^{(j)}\big]^2P_{\ell}^{(j)}\big)^{-1/2}.
\end{equation}
This is what was used in \cite{krdl:1980}, where it was proposed to implement \eqref{eq:krdl1980} via computing
the eigen-decomposition of $\big[P_{\ell}^{(j)}\big]^{\HH}\big[H_{\ell}^{(j)}\big]^2P_{\ell}^{(j)}$. Today,
we know that this is a less way to
compute the orthonormal polar factor of $H_{\ell}^{(j)}P_{\ell}^{(j)}$ because of potential singularity or near singularity.
Kroonenberg and {De Leeuw}~\cite{krdl:1980} needed it for certain continuity while employing a result
from \cite{deso:1959} for some fixed-point iteration to conduct their convergence analysis, but there are still gaps
to fill in their analysis on which we will comment again at the end of this section.
A few comments are in order:
\begin{enumerate}[(i)]
  \item Each $H_{\ell}^{(j)}=[C_{\ell}^{(j)}][C_{\ell}^{(j)}]^{\HH}$ should not be explicitly formed so that matrix-matrix product
        $H_{\ell}^{(j)}P_{\ell}^{(j)}$ can be computed as $[C_{\ell}^{(j)}]([C_{\ell}^{(j)}]^{\HH}P_{\ell}^{(j)})$.
  \item Previous comments on grouping two or more iterations in \Cref{alg:HOOI} to share commonly computed quantities apply here too.
  \item The same remarks on stopping criteria for \Cref{alg:HOOI} also apply here.
\end{enumerate}

We now present a rough estimate of computational complexity per iterative step of ASI, under the assumption that
$k_{i_1}\cdots k_{i_{m-1}}\ll n_{\ell}$, i.e.,
each $C_{\ell}(P_{i_1},\ldots,P_{i_{m-1}})$ in \eqref{eq:Ci} is a tall skinny matrix.
Each iterative step of \Cref{alg:ASI} involves: form $C_{\ell}\equiv C_{\ell}(P_{i_1},\ldots,P_{i_{m-1}})$ evaluated
at most updated $P_{i_1},\ldots,P_{i_{m-1}}$ at the time,
form $[C_{\ell}^{(j)}]([C_{\ell}^{(j)}]^{\HH}P_{\ell}^{(j)})$,
and compute its polar decomposition via the thin SVD, for $1\le \ell\le m$.

Let us start with real tensor $B$, and assume that it is dense.
As in subsection~\ref{ssec:HOOI}, the number of flops for
forming all $C_{\ell}(P_{i_1},\ldots,P_{i_{m-1}})$, in one for-loop: Lines 2--5 of \Cref{alg:ASI},
is also about the same as the one in \eqref{eq:cpx-real-HOOI-Cell}. Forming all
$[C_{\ell}^{(j)}]([C_{\ell}^{(j)}]^{\HH}P_{\ell}^{(j)})$ costs
$$
4k_1k_2\cdots k_m\sum_{\ell=1}^m n_{\ell}
$$
and computing the thin SVD of all $[C_{\ell}^{(j)}]([C_{\ell}^{(j)}]^{\HH}P_{\ell}^{(j)})$ is
at a cost
$$
6\sum_{\ell=1}^m \big[n_{\ell}k_{\ell}^2+20(k_{\ell})^3\big]
$$
flops \cite[p.493]{govl:2013} (see also \cite[subsection 3.2]{li:2024}).
Putting it all together, we find the cost per the for-loop, Lines 2--5 of \Cref{alg:ASI}, is
\begin{equation}\label{eq:cpx-real-ASI}
O\Big(n_1n_2\cdots n_m\sum_{\ell=1}^m k_{\ell}
      +k_1k_2\cdots k_m\sum_{\ell=1}^m n_{\ell}
      +\sum_{\ell=1}^m n_{\ell}k_{\ell}^2\Big)
\end{equation}
flops.
It can be seen that the dominant part is again from forming $C_{\ell}(P_{i_1},\ldots,P_{i_{m-1}})$.
Putting aside the cost for forming $C_{\ell}(P_{i_1},\ldots,P_{i_{m-1}})$,
the remaining part in \eqref{eq:cpx-real-ASI} is likely much smaller than that in \eqref{eq:cpx-real-HOOI} because, e.g.,
for $n_{\ell}=n$ and $k_{\ell}=k\ll n$ for all $\ell$, it is $O(nk^m)$ in \eqref{eq:cpx-real-ASI} against $O(nk^{m+1})$ for the remaining parts in \eqref{eq:cpx-real-HOOI} and \eqref{eq:cpx-real-ASI}.
With that being said, it is noted that per step \Cref{alg:HOOI} should increase the objective more than \Cref{alg:ASI}
if starting at the same approximation $(P_1^{(j)},P_2^{(j)},\ldots,P_m^{(j)})$, a fact that is well
reflected in our numerical experiments in \cref{sec:egs} where it is observed HOOI overwhelmingly takes fewer iterations
than ASI.

If $B$ is complex, then the computational complexity per iterative step is about 4 times as much as
the estimate in \eqref{eq:cpx-real-ASI} (assuming there are about an equal number of addition/subtraction and multiplication operations).

\subsection{Convergence Analysis}\label{ssec:cvg-ASI}
ASI is a special case of the NPDo approach for tensor block-diagonalization \cite{liwy:2026tbd:arXiv}.
\Cref{thm:cvg4ASI} is a corollary of \cite[Theorem 4.1]{liwy:2026tbd:arXiv}, whose proof relies on
\Cref{lm:polar2max,lm:polar2max'}. Recall $C_{\ell}^{(j)}$, $H_{\ell}^{(j)}$, and $\Lambda_{\ell}^{(j)}$
introduced in \eqref{eq:CijHij}.
%

\begin{theorem}\label{thm:cvg4ASI}
Let the sequence $\{(P_1^{(j)},P_2^{(j)},\ldots,P_m^{(j)})\}_{j=0}^{\infty}$ be generated by \Cref{alg:ASI},
$(P_{*1},P_{*2},\ldots,P_{*m})$ be an accumulation point of $\{(P_1^{(j)},P_2^{(j)},\ldots,P_m^{(j)})\}_{j=0}^{\infty}$ whose
subsequence $\{(P_1^{(j)},P_2^{(j)},\ldots,P_m^{(j)})\}_{j\in\bbI}$ converges to $(P_{*1},P_{*2},\ldots,P_{*m})$, and
let $(\what P_{*1},\what P_{*2},\ldots,\what P_{*\,m-1})$ be an accumulation point of $\{(P_1^{(j+1)},P_2^{(j+1)},\ldots,P_{m-1}^{(j+1)})\}_{j\in\bbI}$, and accordingly keep
notation $H_{*\ell}$ and $\what H_{*\ell}$ be as in \eqref{eq:Hell-star-all}.
The following statements hold.
\begin{enumerate}[{\rm (a)}]
  \item The sequence $\{f(P_1^{(j)},P_2^{(j)},\ldots,P_m^{(j)})\}_{j=0}^{\infty}$ is monotonically increasing and convergent;
  \item The equations in \eqref{eq:TD:KKTatMAX'-HOOI} hold and each is a polar decomposition, i.e.,
        $\Lambda_{*\ell}=P_{*\ell}^{\HH}\what H_{*\ell}\,P_{*\ell}\succeq 0$ for $1\le \ell\le m$.

  \item In item~{\rm (b)}, if, for $1\le \ell\le m$,
        \begin{equation}\label{eq:full-rank2}
        \rank(H_{*\ell}P_{*\ell})=k_{\ell},
        \end{equation}
        then $\what P_{*\ell}=P_{*\ell}$ for $1\le \ell\le m-1$, leading to \eqref{eq:TD:KKTatMAX'-HOOI'}
        with all $\Lambda_{*\ell}=P_{*\ell}^{\HH} H_{*\ell}\,P_{*\ell}\succ 0$ for $1\le \ell\le m-1$ and $\Lambda_{*m}\succeq 0$.
  \item If $(P_{*1},P_{*2},\ldots,P_{*m})$ is an isolated accumulation point of $\{(P_1^{(j)},P_2^{(j)},\ldots,P_m^{(j)})\}_{j=0}^{\infty}$
        and if, besides \eqref{eq:full-rank2} for $1\le \ell\le m-1$, also
        \begin{equation}\label{eq:full-rank3}
        \rank(H_{*m}P_{*m})=k_m,
        \end{equation}
        then the entire sequence $\{(P_1^{(j)},P_2^{(j)},\ldots,P_m^{(j)})\}_{j=0}^{\infty}$ converges to $(P_{*1},P_{*2},\ldots,P_{*m})$.

  \item We have $2m$ convergent series, for $1\le \ell\le m$,
        \begin{subequations}\label{eq:cvg4ASI:TD:series}
        \begin{align}
        \sum_{j=0}^{\infty}\sigma_{\min}(H_{\ell}^{(j)}P_{\ell}^{(j)})\,
                         \big\|\sin\Theta\big(\cR(P_{\ell}^{(j+1)}),\cR(P_{\ell}^{(j)})\big)\big\|_{\F}^2
                      &<\infty,    \label{eq:cvg4ASI:TD:series-1} \\
        \sum_{j=0}^{\infty}\sigma_{\min}(H_{\ell}^{(j)}P_{\ell}^{(j)})\,
                  \frac {\big\|H_{\ell}^{(j)}P_{\ell}^{(j)}-P_{\ell}^{(j)}\Lambda_{\ell}^{(j)}\big\|_{\F}^2}
                        {\big\|H_{\ell}^{(j)}\big\|_{\F}^2}
                      &<\infty,              \label{eq:cvg4ASI:TD:series-2}
        \end{align}
        \end{subequations}
        where $\Lambda_{\ell}^{(j)}=\big[P_{\ell}^{(j)}\big]^{\HH}H_{\ell}^{(j)}P_{\ell}^{(j)}\in\bbC^{k_{\ell}\times k_{\ell}}$
\end{enumerate}
\end{theorem}

As a corollary of Theorem~\ref{thm:cvg4ASI}(e),
if $\sigma_{\min}(H_{\ell}^{(j)}P_{\ell}^{(j)})$ is eventually bounded below away from $0$
uniformly,
then
\begin{equation}\label{eq:NEPv-always'}
\lim_{j\to\infty}\frac {\big\|H_{\ell}^{(j)}P_{\ell}^{(j)}-P_{\ell}^{(j)}\Lambda_{\ell}^{(j)}\big\|_{\F}}
                        {\big\|H_{\ell}^{(j)}\big\|_{\F}} =0,
\end{equation}
namely, increasingly $H_{\ell}^{(j)}P_{\ell}^{(j)}\approx P_{\ell}^{(j)}\Lambda_{\ell}^{(j)}$
towards a polar decomposition of $H_{\ell}^{(j)}$. This
means that $(P_1^{(j)},P_2^{(j)},\ldots,P_m^{(j)})$ becomes a more and more accurate approximate solution
to the KKT condition \eqref{eq:KKT-TD},
even in the absence of knowing whether the entire sequence $\{(P_1^{(j)},P_2^{(j)},\ldots,P_m^{(j)})\}_{j=0}^{\infty}$ converges or not.

Comparing \Cref{thm:cvg4ASI} with \Cref{thm:cvg4HOOI}, we point out two subtle differences:
\begin{enumerate}[(1)]
  \item Although items (b,c) of both theorems claim \eqref{eq:TD:KKTatMAX'-HOOI} and \eqref{eq:TD:KKTatMAX'-HOOI'},
        it is only claimed that all $\Lambda_{*\ell}\succeq 0$ $(\succ 0$) in \Cref{thm:cvg4ASI}(b,c)
        whereas in \Cref{thm:cvg4HOOI}(b,c) each $\eig(\Lambda_{*\ell})$ consists of $k_{\ell}$ largest
        eigenvalues of $\what H_{*\ell}\succeq 0$ or $H_{*\ell}\succ 0$, which is stronger than simply $\Lambda_{*\ell}\succeq 0$, that holds by default
        because $\what H_{*\ell}\succeq 0$ and $H_{*\ell}\succ 0$ always;
  \item There is no need to perform an alignment between $P_{\ell}^{(j)}$ and $P_{*\ell}$ in \Cref{thm:cvg4ASI}(d) but there is one in \Cref{thm:cvg4HOOI}(d). This is due to the fact that the orthonormal polar factor is
      unique under the full column-rank assumption.
\end{enumerate}

\begin{remark}\label{rk:ASI-cvg}
Earlier, we pointed out
Kroonenberg and De Leeuw~\cite{krdl:1980} implicitly used orthonormal polar factor in the form \eqref{eq:krdl1980}
(for the case $m=3$) but
without mentioning the term. They also numerically realize \eqref{eq:krdl1980}  through
the eigendecomposition of $\big[P_{\ell}^{(j)}\big]^{\HH}\big[H_{\ell}^{(j)}\big]^2P_{\ell}^{(j)}$.
It can be seen that $P_{\ell}^{(j+1)}$ as in \eqref{eq:krdl1980}
is indeed the orthonormal polar factor of $H_{\ell}^{(j)}P_{\ell}^{(j)}$, and hence the resulting method
is essentially the same method as \Cref{alg:ASI} but can encounter subtle numerical issues
when $\big[P_{\ell}^{(j)}\big]^{\HH}\big[H_{\ell}^{(j)}\big]^2P_{\ell}^{(j)}$ is singular or nearly singular. What we did in \Cref{alg:ASI} is to compute
orthonormal polar factor via thin SVD, completely getting around any numerical singularity issue.
Kroonenberg and De Leeuw~\cite{krdl:1980} conducted their convergence analysis with \eqref{eq:krdl1980} for updating and their conclusions are essentially the ones in \Cref{thm:cvg4ASI}(a,d) under the conditions that all involved matrices $H_{\ell}^{(j)}P_{\ell}^{(j)}$ have full column ranks, which would be satisfied most likely in actual computation
for large scale tensors as often $k_{\ell}\ll n_{\ell}$.
However, carefully examining their proofs reveals
that they actually needed to assume all $H_{\ell}^{(j)}$ to be positive definite, not just each $H_{\ell}^{(j)}P_{\ell}^{(j)}$ having full column rank,
because they used the Cauchy-Schwarz inequality with respect to the
$H_{\ell}^{(j)}$-inner product and hence they needed the positive definiteness of $H_{\ell}^{(j)}$
in order to apply the result of \cite{deso:1959}.
A necessary
condition for $H_{\ell}^{(j)}\succ 0$ is $n_{\ell}\le k_{i_1}\cdots k_{i_{m-1}}$, which is problematic for large scale tensor. To see that, one needs to look at no further
than the extreme case: all $k_{\ell}=1$ and then $\rank(H_{\ell}^{(j)})=1$ always
and hence $H_{\ell}^{(j)}$ is singular if $n_{\ell}>1$.
Lastly, \Cref{thm:cvg4ASI} does not distinguish real or complex tensors and contains much deeper results, for example,
the convergence series in item (d).
\end{remark}


\section{Numerical Experiments}\label{sec:egs}
In this section, we will report some numerical tests to illustrate the performance of both HOOI and ASI. We are especially
interested in the situations where
$$
k_{\ell}\le k_{i_1}\cdots k_{i_{m-1}}\ll n_{\ell}.
$$
As we commented earlier for \Cref{alg:HOOI}, the first condition makes sure that $C_{\ell}^{(j)}$ has a nontrivial left singular subspace of dimension $k_{\ell}$ while the second condition makes MATLAB's {\tt svd} is a suitable choice
to compute  $P_{\ell}^{(j+1)}$. If, however, $k_{i_1}\cdots k_{i_{m-1}}\not\ll n_{\ell}$ and $n_{\ell}$ is large (in thousands
or larger), some iterative methods may have to be called.

As $m>3$ introduces no essentially difference in theory and experiments, except perhaps more computational work, we will use $3$-mode tensor in our experiments, i.e., $m=3$.
We will start with the Miranda scientifc similation tensor
\cite{bako:2025} and then more extensively with randomly generated tensors that nearly fit into the TD model, in the sense of \eqref{eq:hatB}.

Methods included in our experiments are HOOI (\Cref{alg:HOOI}) and ASI (\Cref{alg:ASI}) with the same random initial guess
or with one by truncated HOSVD \cite{dldv:2000a} that essentially Tucker~\cite{tuck:1966} used as the solution to
the decomposition that bears his name nowadays. Our random initial guess is generated as, for $1\le\ell\le m$,
\begin{equation}\label{eq:init(random)}
P_{\ell}={\tt orth}({\tt randn}(n_{\ell},k_{\ell}))
\quad\mbox{or}\quad{\tt orth}({\tt randn}(n_{\ell},k_{\ell})+{\tt 1i*randn}(n_{\ell},k_{\ell})),
\end{equation}
dependent on real or complex tensors, where {\tt randn} and {\tt orth} are two MATLAB functions that
return a random matrix and an orthonormal matrix, respectively. The initial guess by truncated HOSVD is obtained by, for $1\le\ell\le m$,
\begin{equation}\label{eq:init(HOSVD)}
\mbox{economic SVD:}\,\, B_{\ufd,\ell}=U_{\ell}\Sigma_{\ell} V_{\ell}^{\HH}, \quad
P_{\ell}=[U_{\ell}]_{(:,1:k_{\ell})},
\end{equation}
where $[U_{\ell}]_{(:,1:k_{\ell})}$ refers to columns 1 to $k_{\ell}$ of $U_{\ell}$, which can also, and has to, be computed
iteratively if both $n_{\ell}$ and $N/n_{\ell}$ are large (in thousands or larger).
We also refer such $(P_1,\ldots,P_m)$ as the TD solution by HOSVD.

All experiments are carried out within the MATLAB environment (MATLAB R2022) on a Dell Precision 3660 desktop with an
Intel i9 processor (3200 Mhz), 32 GB memory, running Microsoft Windows 11 Enterprise.

\begin{figure}[t]
{\centering
\begin{tabular}{ccc}
  \resizebox*{0.31\textwidth}{0.17\textheight}{\includegraphics{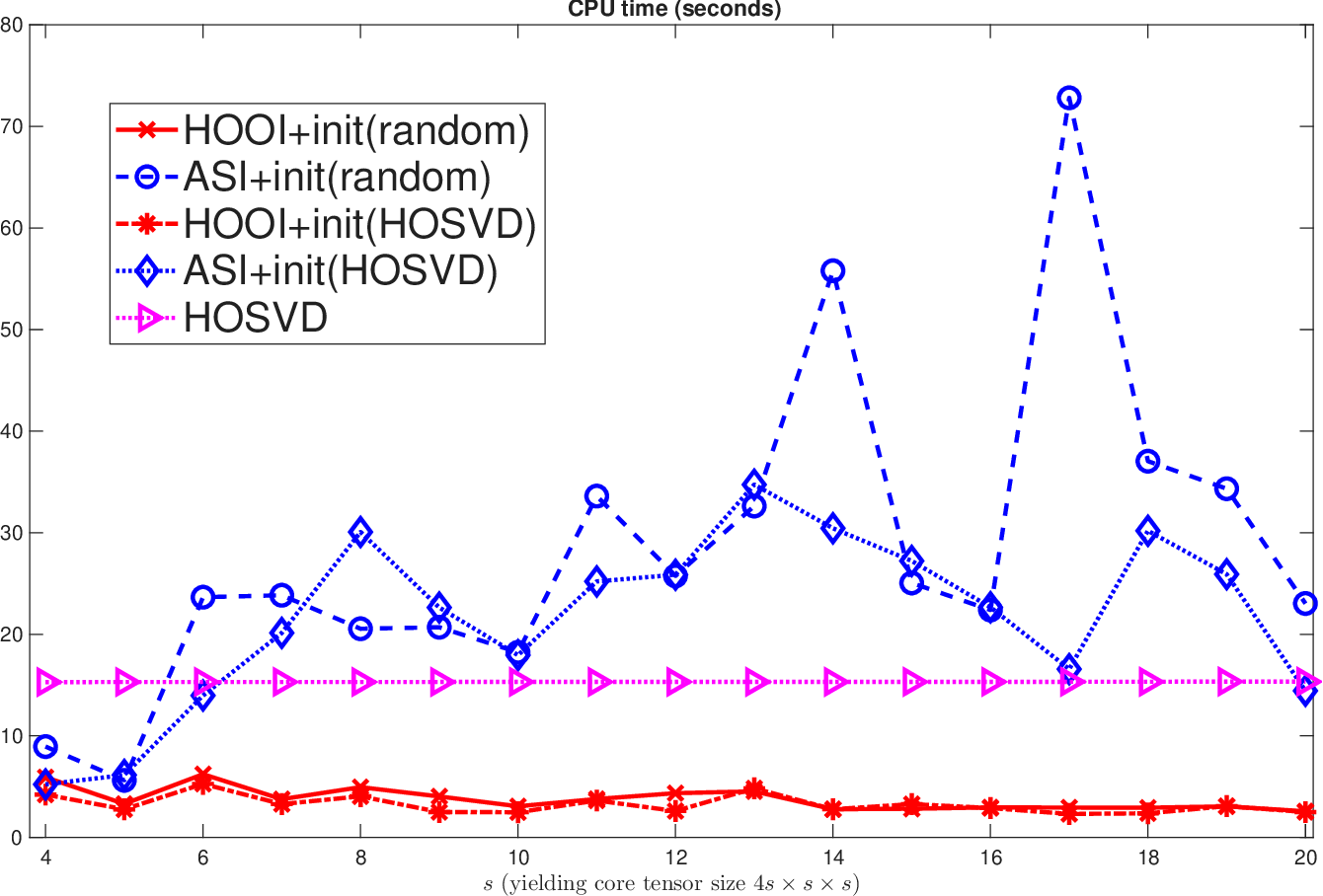}}
  &\resizebox*{0.31\textwidth}{0.17\textheight}{\includegraphics{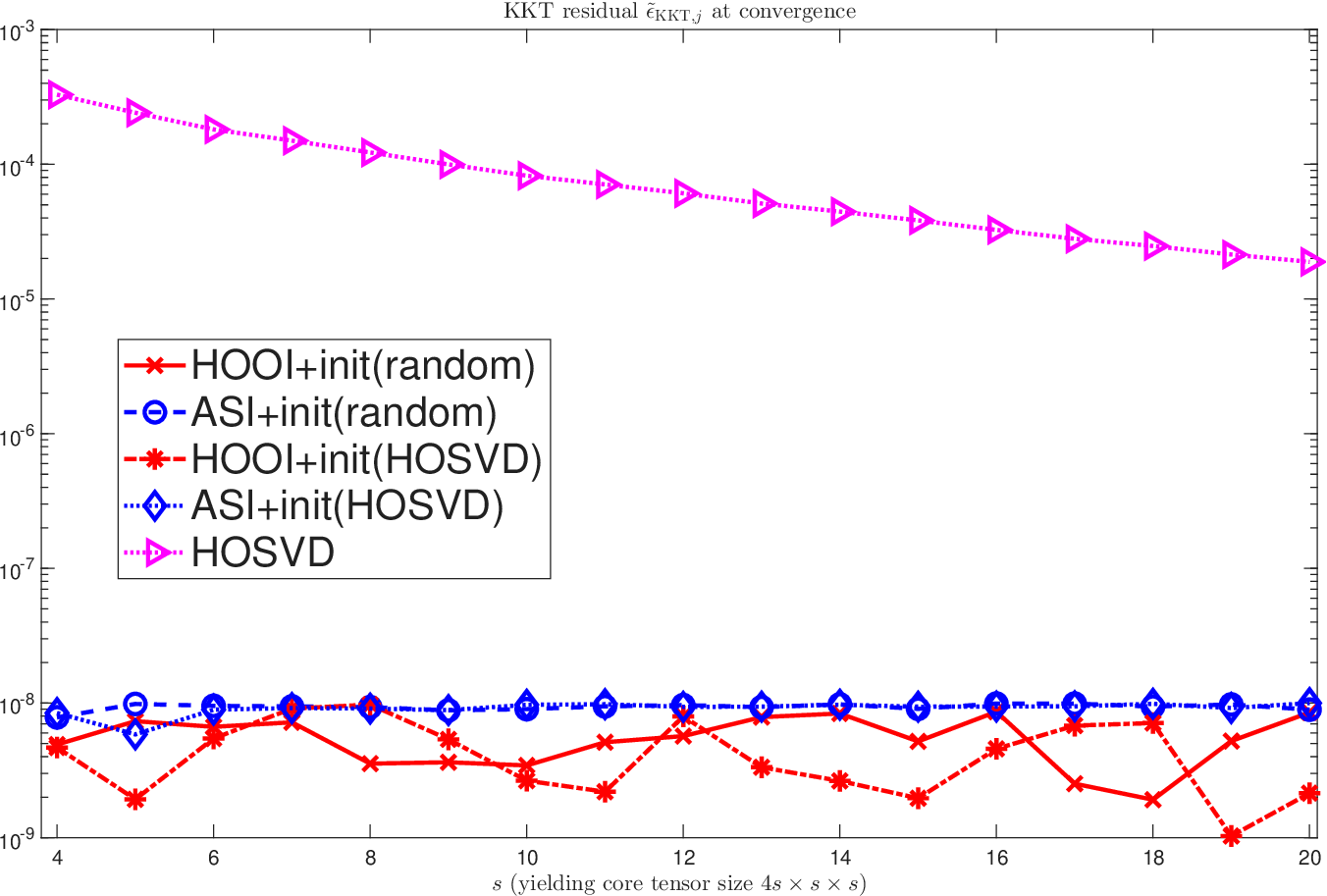}}
  &\resizebox*{0.31\textwidth}{0.17\textheight}{\includegraphics{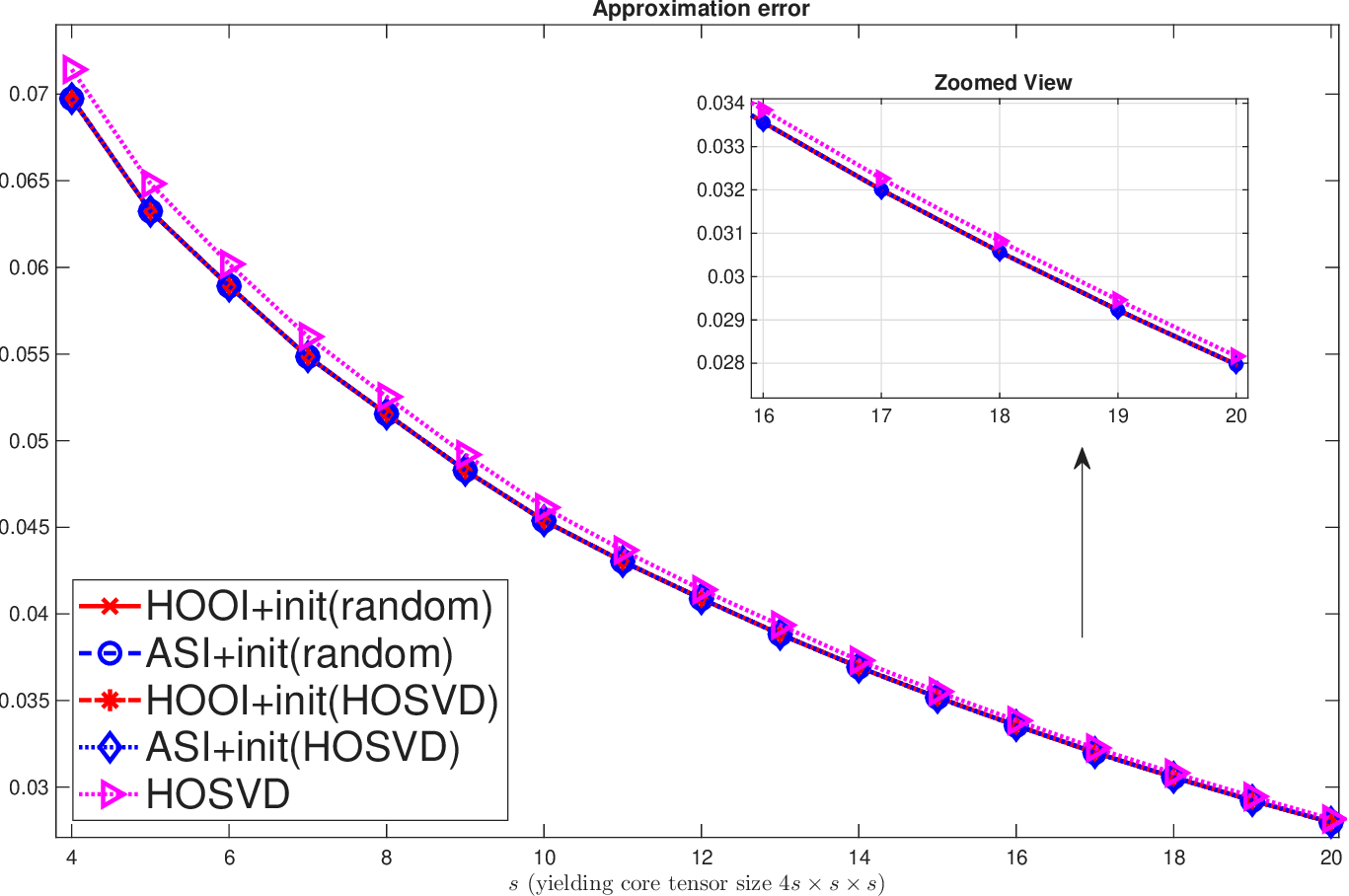}}
\end{tabular}\par
}
\vspace{-0.15 cm}
\caption{\small Performance statistics on the Miranda tensor:
  core tensor size $4s\times s\times s$ for $4\le s\le 20$.
  {\em Left:} CPU time; {\em Middle:}  KKT residual $\tilde\epsilon_{\KKT,j}$ at convergence;
  {\em Right:} approximation error.
  By ``init(random)'', it means that random initial guess
  \eqref{eq:init(random)} is used and similarly ``init(HOSVD)'' means initial guess by \eqref{eq:init(HOSVD)}.
  The CPU time for ``HOOI+init(HOSVD)'' and ``ASI+init(HOSVD)'' does not include the time used for computing initial guess by HOSVD.
  }
\label{fig:Miranda:stat}
\end{figure}

\begin{figure}
{\centering
\begin{tabular}{ccc}
  \resizebox*{0.31\textwidth}{0.17\textheight}{\includegraphics{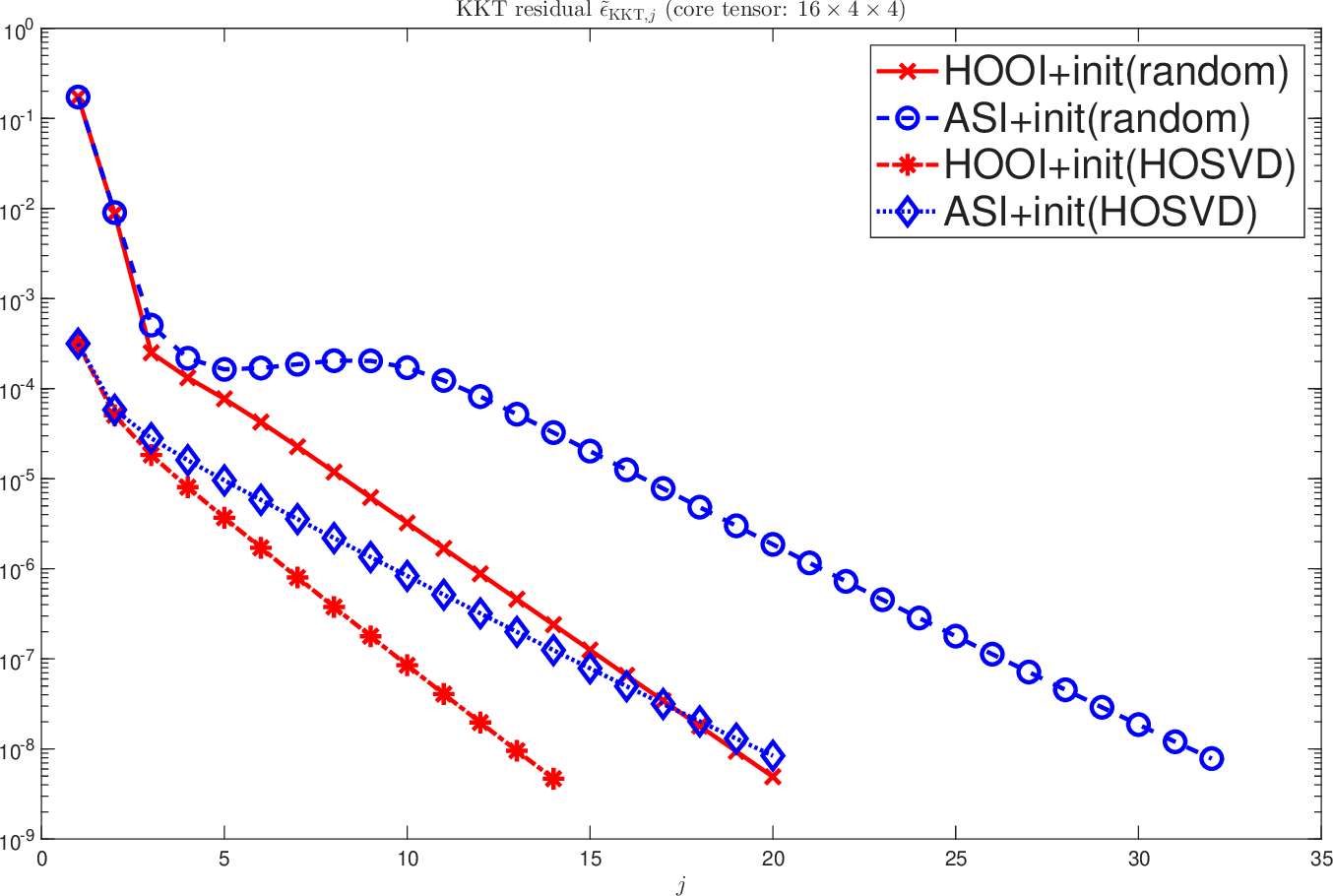}}
  &\resizebox*{0.31\textwidth}{0.17\textheight}{\includegraphics{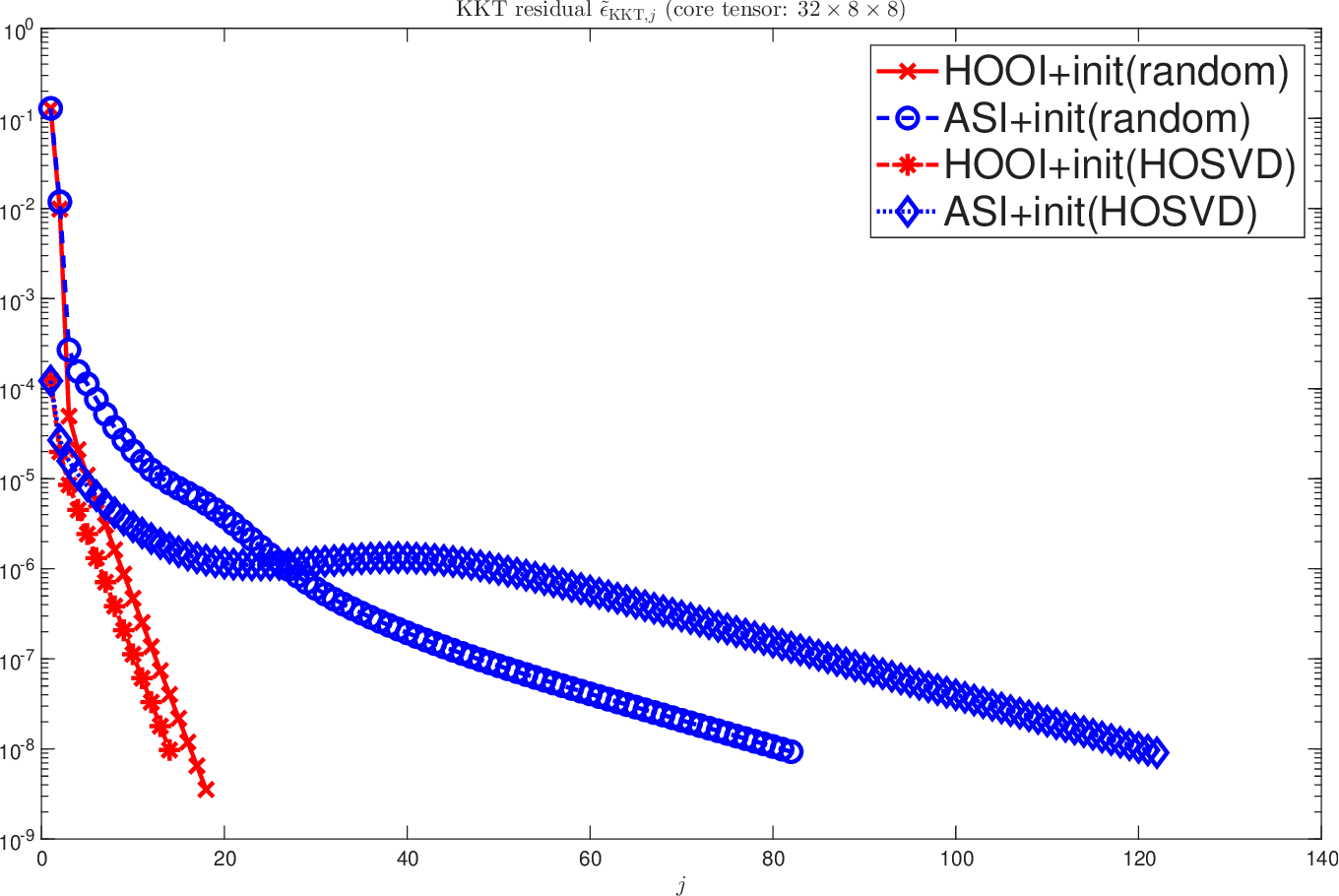}}
  &\resizebox*{0.31\textwidth}{0.17\textheight}{\includegraphics{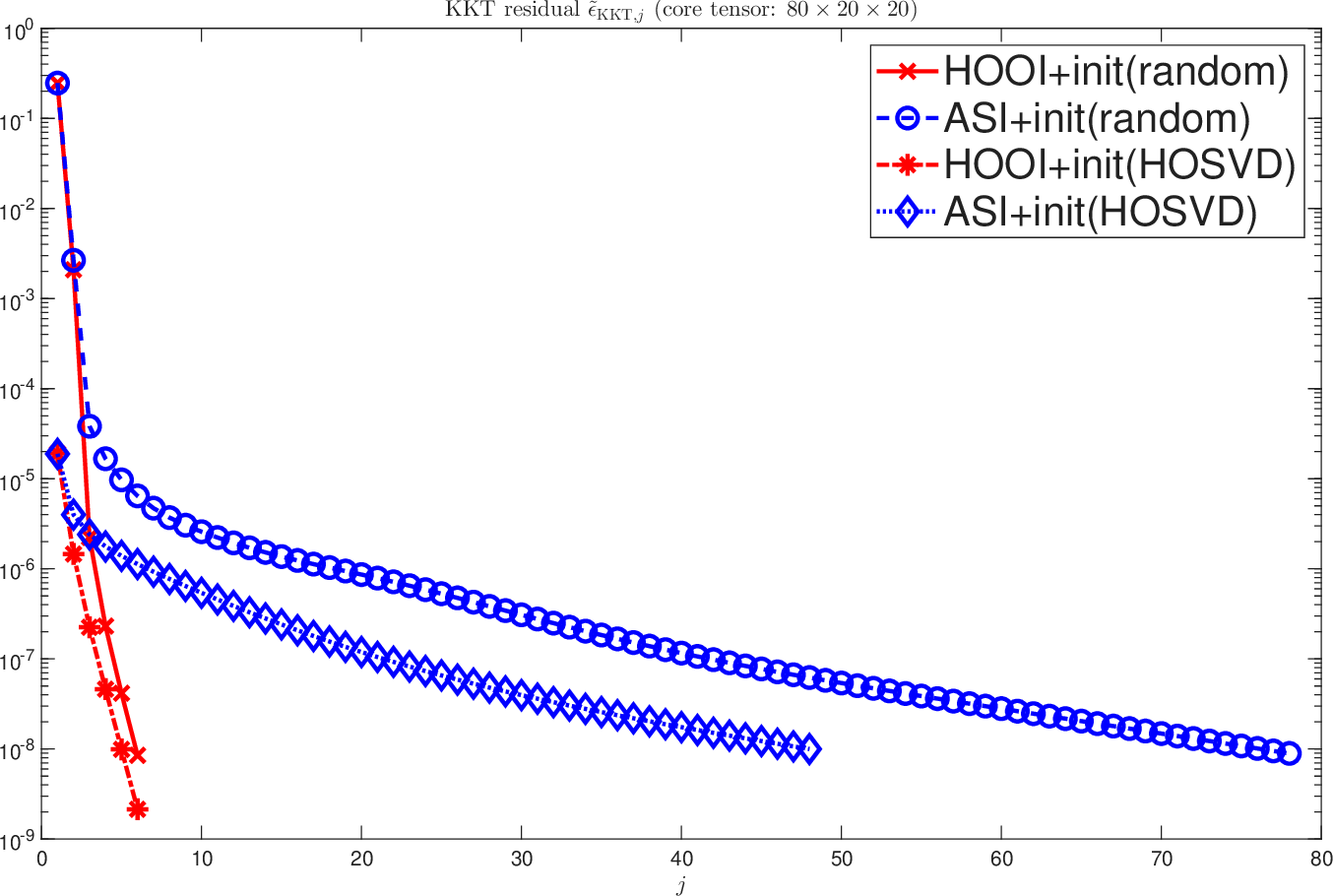}}
\end{tabular}\par
}
\vspace{-0.15 cm}
\caption{\small Convergence behavior in terms of normalized KKT residual $\tilde\epsilon_{\KKT,j}$ on the Miranda tensor:
  core tensor size $4s\times s\times s$ for $s=4, 8, 20$. It is interesting to notice that for $s=8$,
  ``ASI+init(HOSVD)'' starts out with a much better initial guess than ``ASI+init(random)'', but
  is overtaken by the latter at iteration 27 and ends up with taking 82 iterations while
  ``ASI+init(random)'' takes 122 iterations.
  }
\label{fig:Miranda:cvg}
\end{figure}

\subsection{Miranda Scientifc Similation Tensor}
The Miranda scientifc similation tensor $B$ is a real tensor of size $2048\times 256\times 256$ and was used in \cite{bako:2025} for illustrations. It is available online and its detail can be found \cite{bako:2025}. Given its
$n_1=2048=8n_2=8n_3$, we will attempt to compute its optimal TD with core tensor size $4s\times s\times s$
with $s$ varying from $4$ to $20$. In \Cref{fig:Miranda:stat}, we show the statistics of CPU time, final  KKT residual $\tilde\epsilon_{\KKT,j}$,
and approximation error defined by
\begin{equation}\label{eq:approxerr}
\frac {\|B-\what B\|_{\F}}{\|B\|_{\F}}
\end{equation}
as $s$ varies from 4 to 20, where $\what B$ is a computed TD in the form of \eqref{eq:hatB}. We observed the following:
\begin{enumerate}[(1)]
  \item KKT residual $\tilde\epsilon_{\KKT,j}$ reaches $\epsilon_2=10^{-8}$ in \eqref{eq:KKT-stop'}
        for both HOOI and ASI, while it stays at an elevated level for truncated HOSVD, meaning
        the solution simply by HOSVD is far from minimizing approximation error \eqref{eq:approxerr},
        as it is clear from the third plot in \Cref{fig:Miranda:stat}.
  \item CPU time by HOSVD stays constant because it is done by calling MATLAB's economic SVD \eqref{eq:init(HOSVD)} separately for each unfolding matrix.
  \item HOOI is faster than ASI, regardless initial guess -- randomly generated or by HOSVD.
        In comparison to ASI, HOOI computes each approximation to $P_{\ell}$ optimally
        by the eigen-decomposition of $H_{\ell}=C_{\ell}C_{\ell}^{\HH}$ in \eqref{eq:Hi} whereas ASI essentially does one step of
        the power iteration with $H_{\ell}$; in comparison to HOSVD, each $C_{\ell}$ is much skinnier
        than the unfolding matrix $B_{\ufd,\ell}$ due to $k_{\ell}\ll n_{\ell}$ and thus
        computing the economic SVD in ASI
        is so much faster than the latter.
  \item In general, initial guess by HOSVD
        is a quality one, but it may not necessarily lead to a fewer number of iterations occasionally. We can read a few cases of such things for ASI from the first plot in \Cref{fig:Miranda:stat}, e.g., $s=8$ (see also the second plot in  \Cref{fig:Miranda:cvg}).    Also, TD simply by truncated HOSVD is far from optimal as we commented moments ago.
\end{enumerate}
Next, \Cref{fig:Miranda:cvg} plots the convergence behaviors by the methods in terms of
the normalized KKT residual $\tilde\epsilon_{\KKT,j}$ for $s=4, 8, 20$. These behaviors
are selected to be representative: HOOI takes the fewer numbers of iterations and using the initial by the truncated HOSVD in general cuts down the number of iterations by both HOOI and ASI. It is noted that, with the initial by HOSVD,
the starting KKT residual $\epsilon_{\KKT,0}$ is much smaller, meaning the initial guess is a quality one.
On the other hand, it is noted from the second plot for $s=8$ in \Cref{fig:Miranda:cvg} that
the initial by HOSVD may not pay off occasionally.

\subsection{Randomly Generated Tensors}
Given $n_{\ell}$ and $k_{\ell}$ for $1\le\ell\le m=3$, we first generate, in MATLAB,
dependent on real or complex tensors,
\begin{alignat*}{2}
T&={\tt zeros}([n_1,n_2,n_3]), && \\
T_{(1:k_1,1:k_2,1:k_3)}&={\tt randn}([k_1,k_2,k_3])\quad
             &&\mbox{or}\quad{\tt randn}([k_1,k_2,k_3])+{\tt 1i*randn}([k_1,k_2,k_3]), \\
E&={\tt randn}([n_1,n_2,n_3])\quad
             &&\mbox{or}\quad {\tt randn}([n_1,n_2,n_3])+{\tt 1i*randn}([n_1,n_2,n_3]),
\end{alignat*}
and, for $1\le\ell\le m=3$,
$$
Q_{\ell}={\tt orth}({\tt randn}(n_{\ell})),
\quad\mbox{or}\quad
  {\tt orth}({\tt randn}(n_{\ell})+{\tt 1i*randn}(n_{\ell})),
$$
and finally
$$
B=(T+\eta\,E)\times_1 Q_1\times_2 Q_2\times_3 Q_3,
$$
where parameter $\eta$ controls how close the generated $B$ to the exact TD model. In particular, for $\eta=0$,
$B=T\times_1 P_1\times_2 P_2\times_3 P_3$ exactly for some $T\in\bbC^{k_1\times k_2\times k_3}$ and
$P_{\ell}\in\STM{k_{\ell}}{n_{\ell}}$ for $1\le\ell\le m=3$.
The same random initial \eqref{eq:init(random)} for each generated $B$ is used throughout.

\begin{figure}[t]
{\centering
\begin{tabular}{c|ccc}
&$\eta=2^{-3}$ & $\eta=2^{-4}$ & $\eta=2^{-5}$ \\  \hline
\multirow{ 2}{*}{\rotatebox{90}{\hspace*{1.2cm}real}}
  &\resizebox*{0.28\textwidth}{0.15\textheight}{\includegraphics{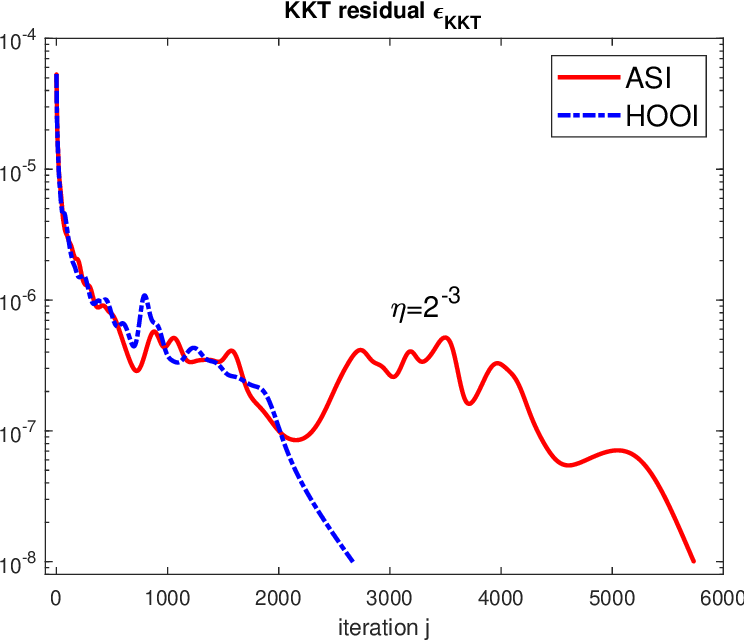}}
  & \resizebox*{0.28\textwidth}{0.15\textheight}{\includegraphics{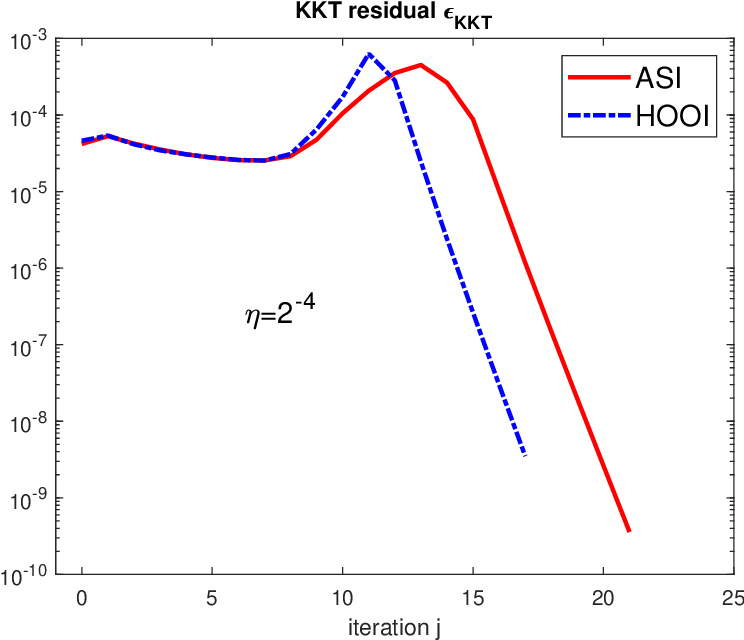}}
  & \resizebox*{0.28\textwidth}{0.15\textheight}{\includegraphics{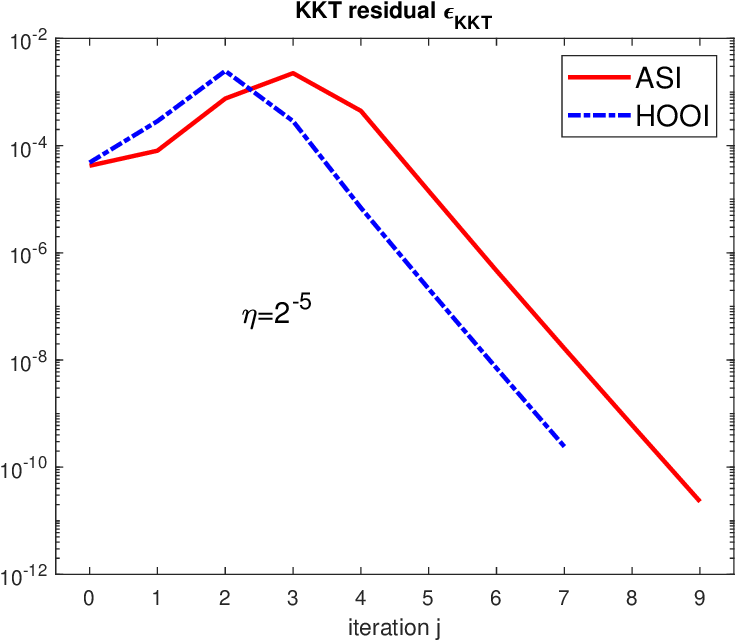}}
      \vphantom{\vrule height94pt width0pt depth0pt}\\
  &\resizebox*{0.28\textwidth}{0.15\textheight}{\includegraphics{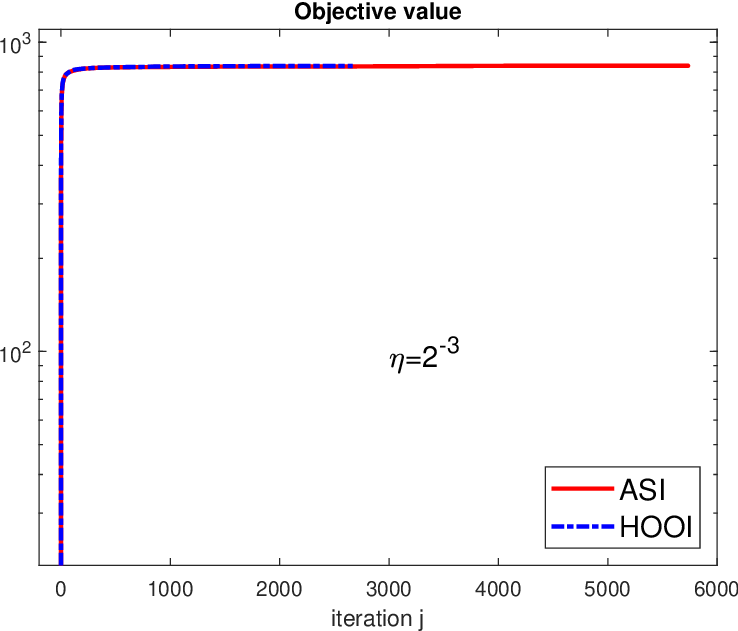}}
  & \resizebox*{0.28\textwidth}{0.15\textheight}{\includegraphics{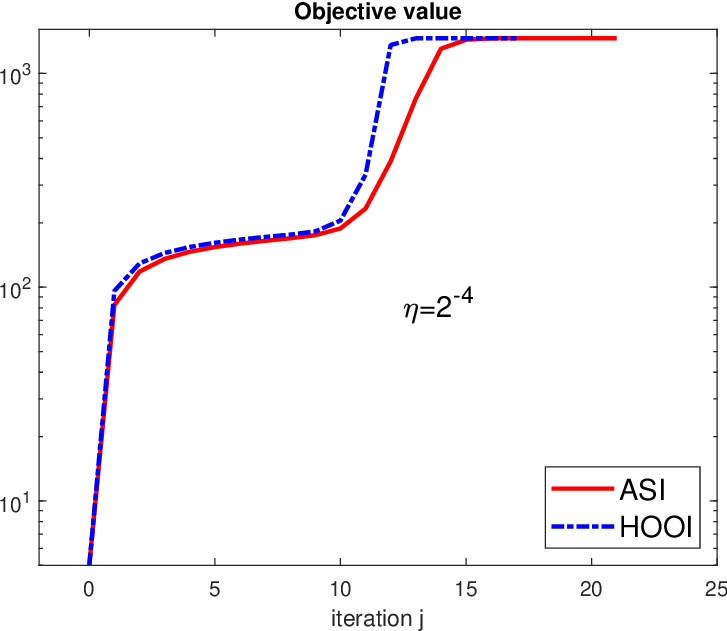}}
  & \resizebox*{0.28\textwidth}{0.15\textheight}{\includegraphics{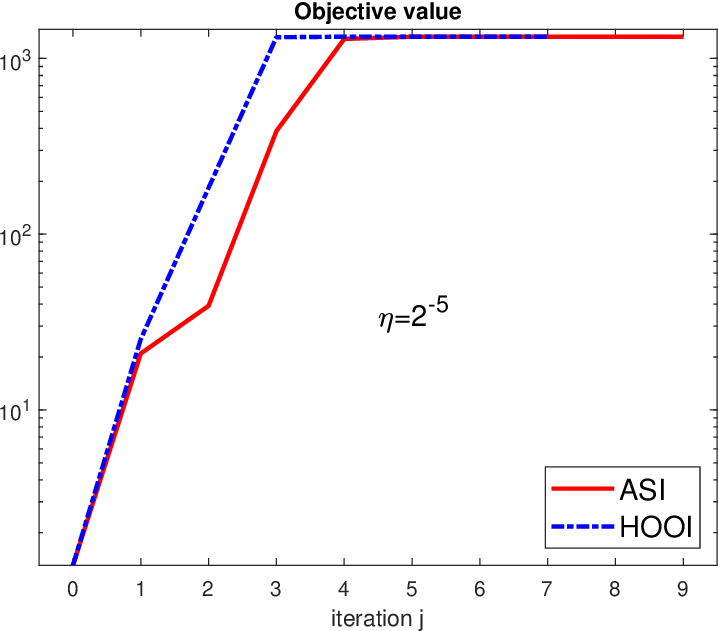}} \\ \hline
\multirow{ 2}{*}{\rotatebox{90}{\hspace*{1.4cm}complex}}
  &\resizebox*{0.28\textwidth}{0.15\textheight}{\includegraphics{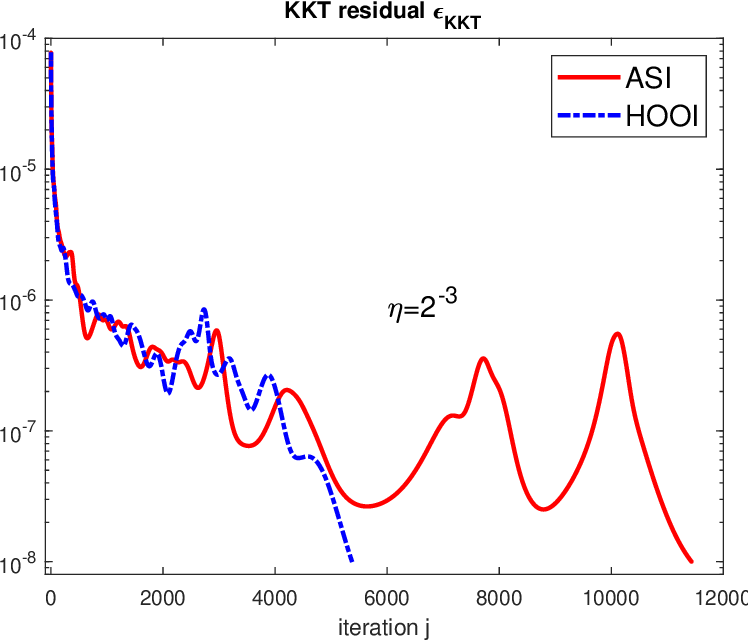}}
  & \resizebox*{0.28\textwidth}{0.15\textheight}{\includegraphics{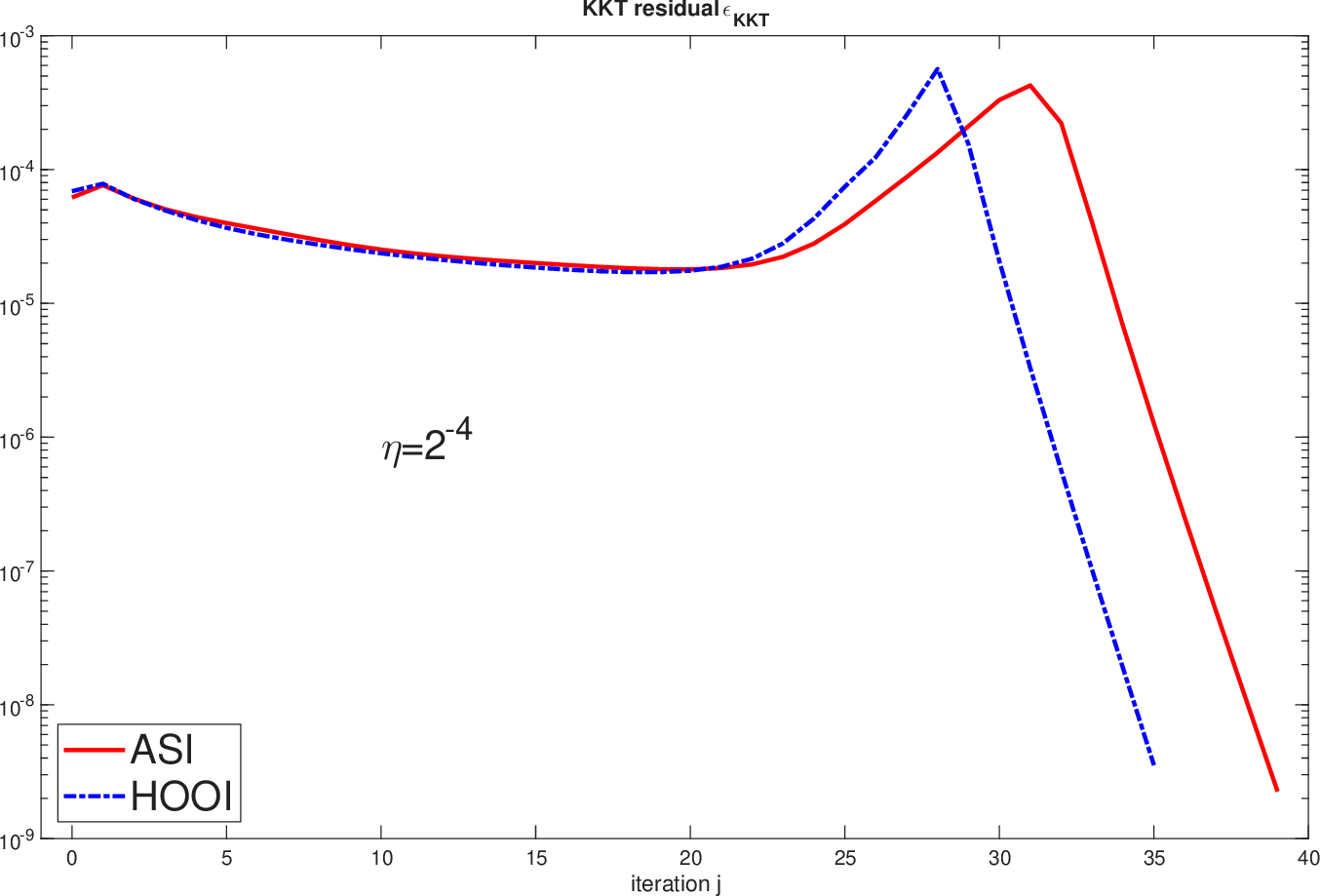}}
  & \resizebox*{0.28\textwidth}{0.15\textheight}{\includegraphics{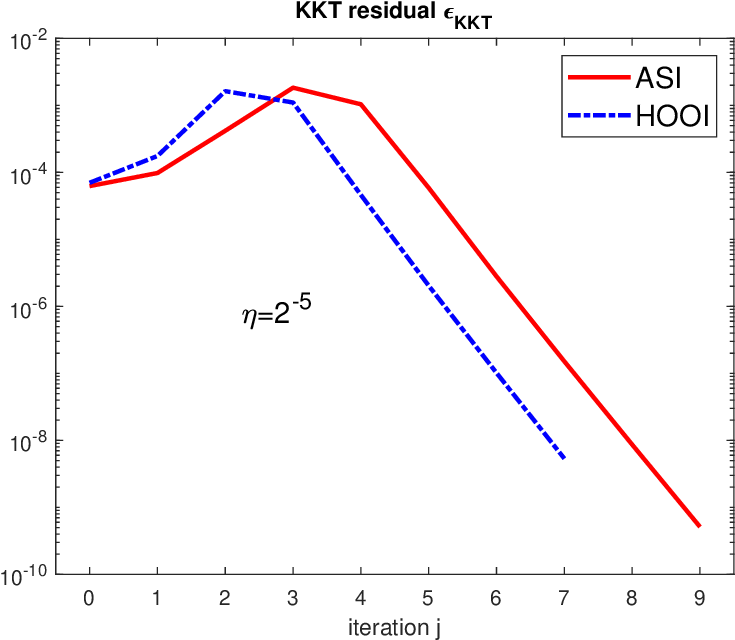}}
      \vphantom{\vrule height94pt width0pt depth0pt}\\
  &\resizebox*{0.28\textwidth}{0.15\textheight}{\includegraphics{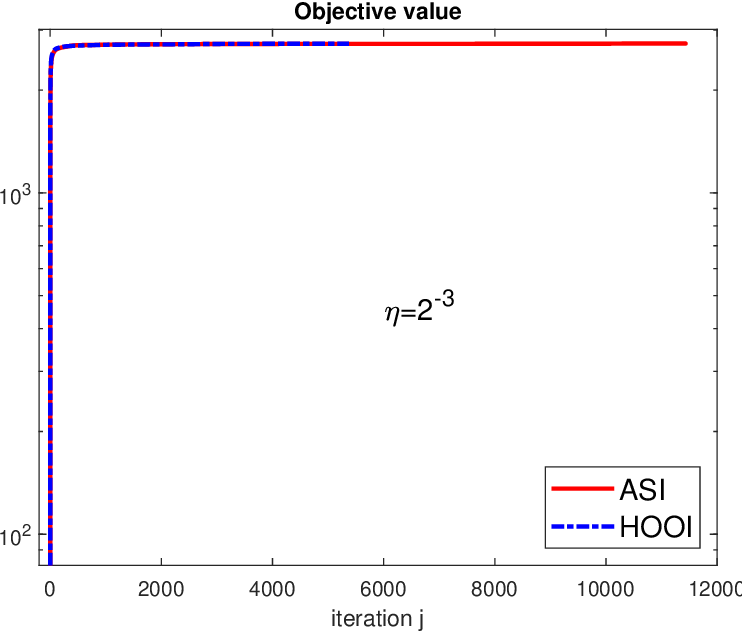}}
  & \resizebox*{0.28\textwidth}{0.15\textheight}{\includegraphics{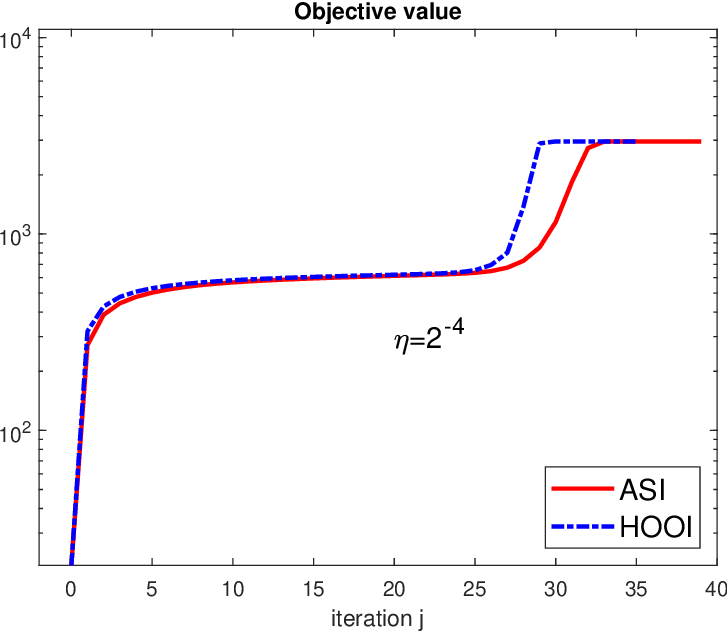}}
  & \resizebox*{0.28\textwidth}{0.15\textheight}{\includegraphics{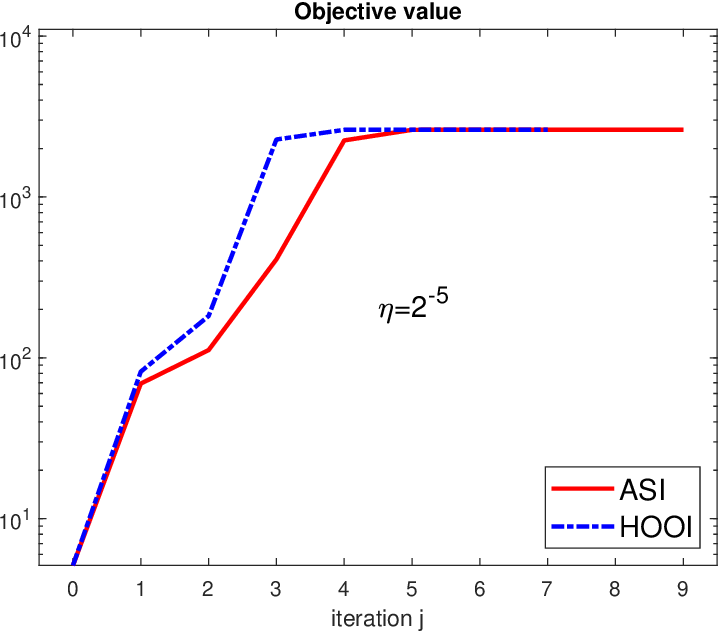}} 
\end{tabular}\par
}
\vspace{-0.15 cm}
\caption{\small Convergence in terms of KKT residual $\tilde\epsilon_{\KKT,j}$ as defined in \eqref{eq:KKT-stop'} and
         objective value on real tensors with $[n_1,n_2,n_3]=[600,   550,   500]$
    and complex tensors with $[n_1,n_2,n_3]=[480,   440,   400]$, and $[k_1,k_2,k_3]=[12,11,10]$,
    as $\eta$ varies in $\{2^{-3},2^{-4}, 2^{-5}\}$. In each case: real or complex, the first row is
    for KKT residual $\tilde\epsilon_{\KKT,j}$ while the second row is for objective value.
  }
\label{fig:behavior-cvg}
\end{figure}

\subsubsection{Convergence behavior}
We randomly generate 8 real and complex tensors, respectively, as $\eta$ varies from $2^{-3}$ down to $2^{-5}$ and save them for experiments. Their dimensions are
$$
\mbox{real:}\quad [n_1,n_2,n_3]=[600,   550,   500],
\quad\mbox{and}\,\,\mbox{complex:}\quad [n_1,n_2,n_3]=[480,   440,   400]
$$
and $[k_1,k_2,k_3]=[12,11,10]$. These size parameters almost reach the limit of MATLAB's {\tt mat} file data format for saving.
Without saving a tensor after its generation, we can go for larger sizes as we will do in the next subsection.

\Cref{fig:behavior-cvg} displays the convergence history in terms of KKT residuals $\tilde\epsilon_{\KKT,j}$ as defined in \eqref{eq:KKT-stop'}
and objective values against the iteration index $j$. We observed the following:
\begin{enumerate}[(1)]
  \item The objective value monotonically increases and eventually flats out at convergence.
  \item The overall trend of the KKT residual $\tilde\epsilon_{\KKT,j}$ is moving towards $0$, although it is not monotonic.
  \item The underlying problem of computing optimal approximate TD gets easier as $\eta$ decreases, in terms of rapid convergence
        toward a (local) optima.
  \item HOOI turns to take fewer iterations than ASI. This can be conceivably explained because at each sub-step
        on $H_{\ell}^{(j)}$,  HOOI  computes the best approximation from its top $k_{\ell}$ eigenvectors whereas
        ASI does one step of the subspace (power) iteration.
  \item It is interesting to note that, for both the real and complex cases, at $\eta=2^{-4}$, just when the objective value
        begins to settle down, as if convergence has already occurred, it moves sharply up to the next level. This legitimizes
        our earlier worry that the stopping criterion \eqref{eq:obj-stop} may end the computation prematurely.
\end{enumerate}
We also tested with $\eta=2^{-1}$ and $2^{-2}$. Similarly behaviors as for $\eta=2^{-3}$ are observed, except more iterations
and thus taking (much) longer to finish.

\subsubsection{Scalability}
\begin{figure}[t]
{\centering
\begin{tabular}{cccc}
\rotatebox{90}{\hspace*{1.2cm}$\eta=2^{-3}$}
  &\resizebox*{0.28\textwidth}{0.15\textheight}{\includegraphics{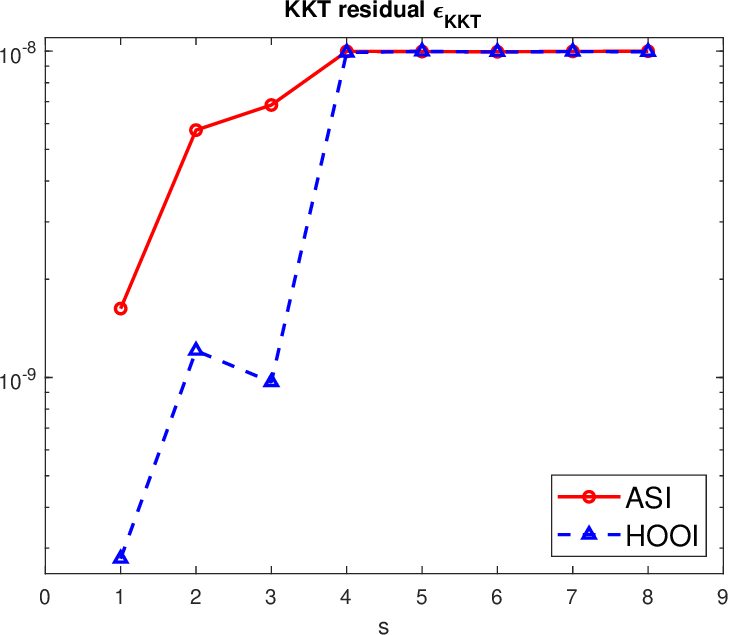}}
  &\resizebox*{0.28\textwidth}{0.15\textheight}{\includegraphics{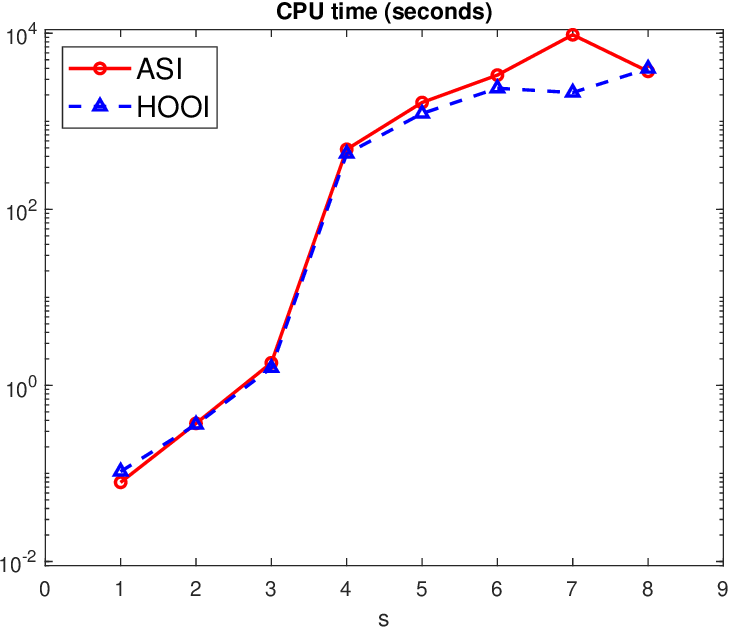}}
  &\resizebox*{0.28\textwidth}{0.15\textheight}{\includegraphics{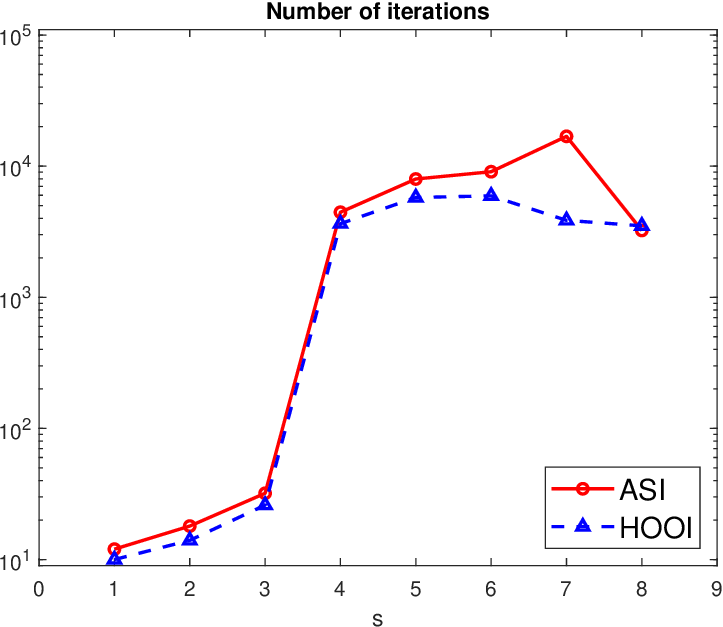}} \\
%
\rotatebox{90}{\hspace*{1.2cm}$\eta=2^{-4}$}
  &\resizebox*{0.28\textwidth}{0.15\textheight}{\includegraphics{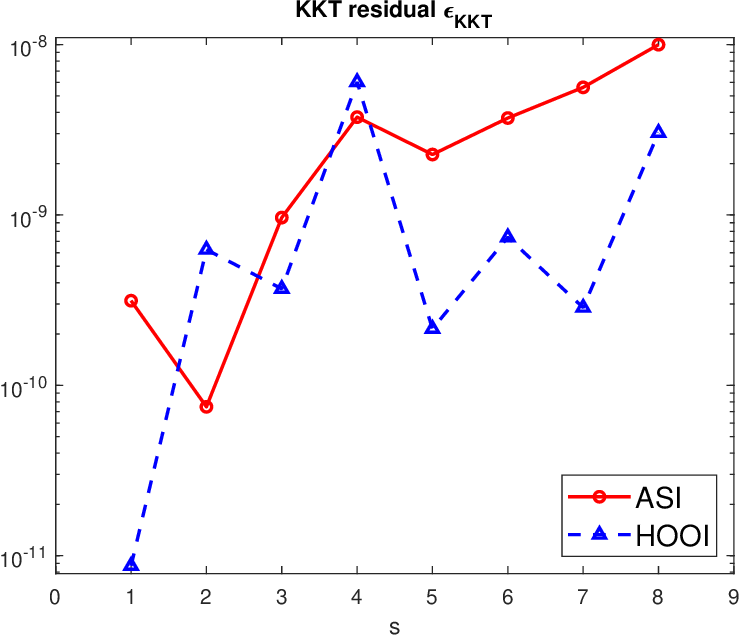}}
  &\resizebox*{0.28\textwidth}{0.15\textheight}{\includegraphics{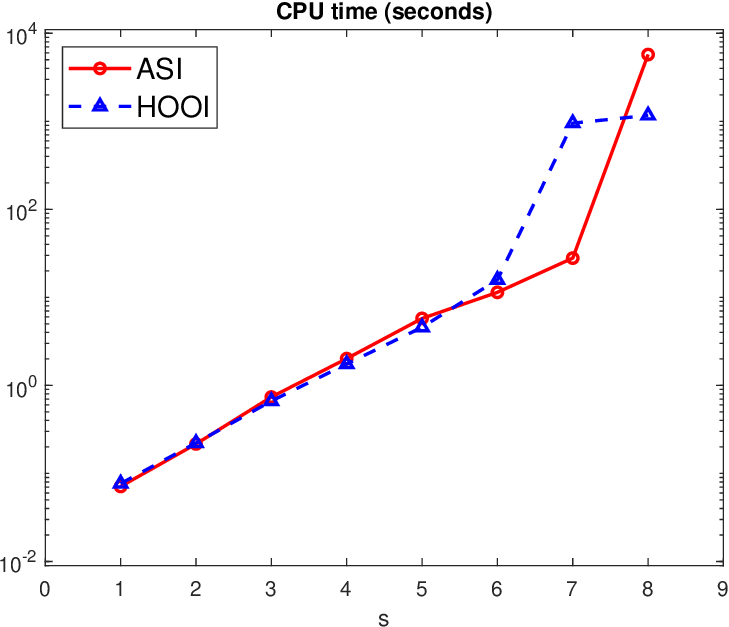}}
  &\resizebox*{0.28\textwidth}{0.15\textheight}{\includegraphics{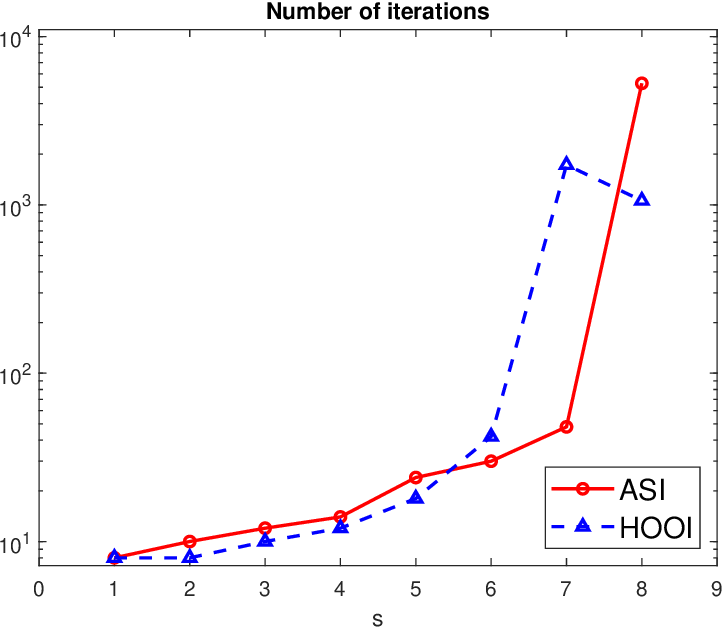}} \\
\rotatebox{90}{\hspace*{1.2cm}$\eta=2^{-5}$}
  &\resizebox*{0.28\textwidth}{0.15\textheight}{\includegraphics{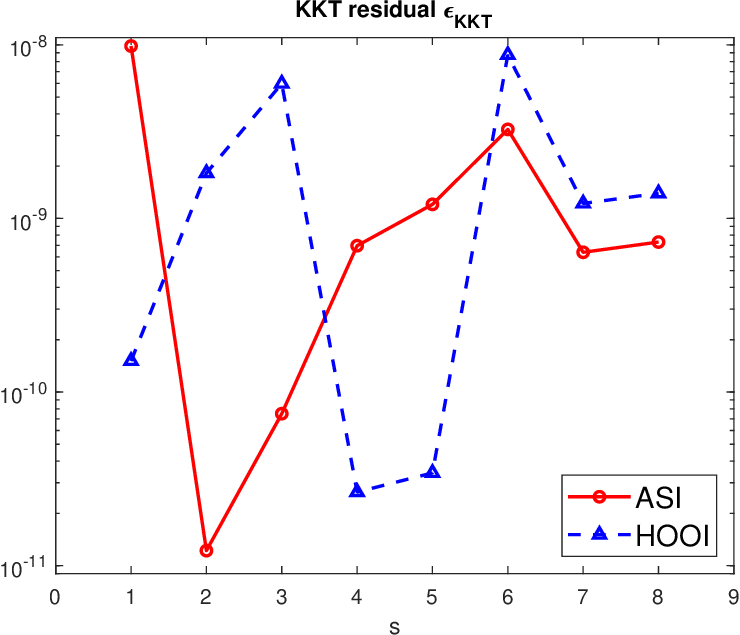}}
  &\resizebox*{0.28\textwidth}{0.15\textheight}{\includegraphics{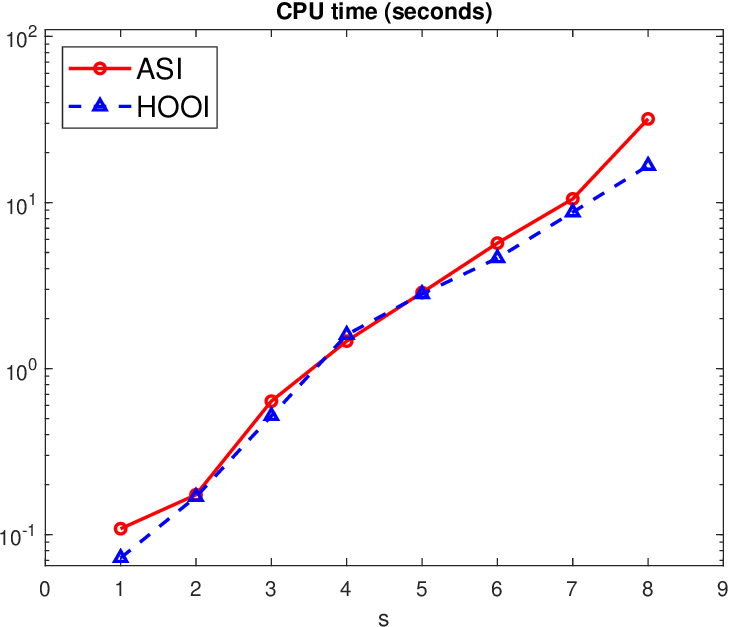}}
  &\resizebox*{0.28\textwidth}{0.15\textheight}{\includegraphics{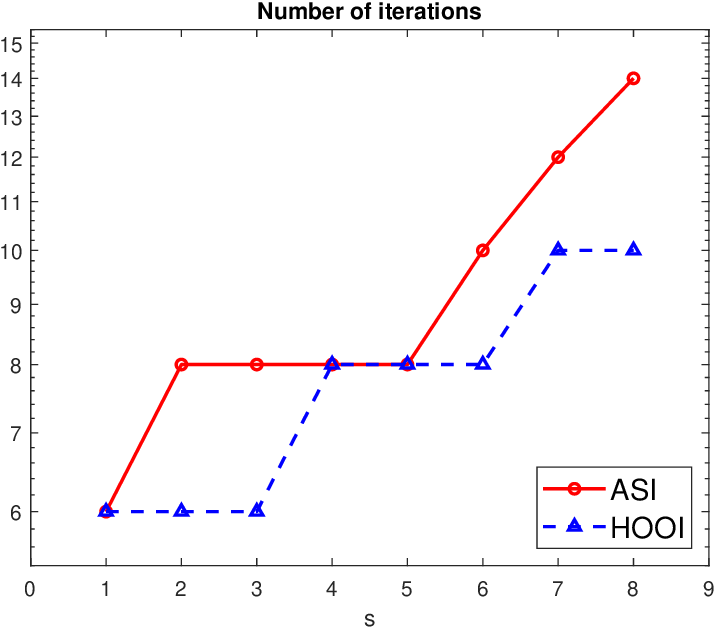}}
\end{tabular}\par
}
\vspace{-0.15 cm}
\caption{\small Scalability of HOOI and ASI on real tensors with $[n_1,n_2,n_3]$ varying according to \eqref{eq:tensor-sizes}
    for $1\le s\le 8$ and $[k_1,k_2,k_3]=[12,11,10]$. {\em Left panel:\/} KKT residual $\tilde\epsilon_{\KKT}$,
    {\em Middle panel:\/} CPU time, and {\em Right panel:\/} the number of iterations.
  }
\label{fig:scale-real}
\end{figure}
For this experiment, we run tests on randomly generated tensors without saving them
and thus we can go for tensors of larger sizes than the ones in the previous subsection.
We still use $[k_1,k_2,k_3]=[12,11,10]$ but let
\begin{equation}\label{eq:tensor-sizes}
[n_1,n_2,n_3]=s\cdot [100, 110, 120]\quad\mbox{for $s=1,2,\ldots, 8$}.
\end{equation}
The last few $s$ yield larger tensors.
In \Cref{fig:scale-real,fig:scale-cmpx} we plot the final KKT residual $\tilde\epsilon_{\KKT,j}$, CPU time, and the number of iterations
as $s$ varies from 1 to 8, for real and complex tensors, respectively. We summarize our observations
as follows:
\begin{enumerate}[(1)]
  \item Both methods converge with final KKT residual $\tilde\epsilon_{\KKT,j}$ about $10^{-8}$ or smaller, and as $\eta$ gets smaller,
        the random TD problems become easier for both methods in terms of the numbers of iterations required and
        consumed CPU time.
  \item Overall HOOI turns to need a fewer number of iterations than ASI does.
  \item HOOI and ASI are competitive to each other in terms of CPU time. The reason
        is $k_{i_1}k_{i_2}=100\ll n_{\ell}$ for $s\ge 2$, and we use the thin SVD on each $C_{\ell}^{(j)}$.
        At each iterative steps HOOI computes three economic SVD of $C_{\ell}^{(j)}\in\bbC^{n_{\ell}\times k_{i_1}k_{i_2}}$
        and ASI also computes three economic SVD of $H_{\ell}^{(j)}P_{\ell}^{(j)}\in\bbC^{n_{\ell}\times k_{\ell}}$,
        which is smaller in size than $C_{\ell}^{(j)}$. The saving from fewer numbers of iterations for HOOI
        and that from smaller sized SVD for ASI seem to balanced out.
\end{enumerate}

\begin{figure}[t]
{\centering
\begin{tabular}{cccc}
\rotatebox{90}{\hspace*{1.2cm}$\eta=2^{-3}$}
  &\resizebox*{0.28\textwidth}{0.15\textheight}{\includegraphics{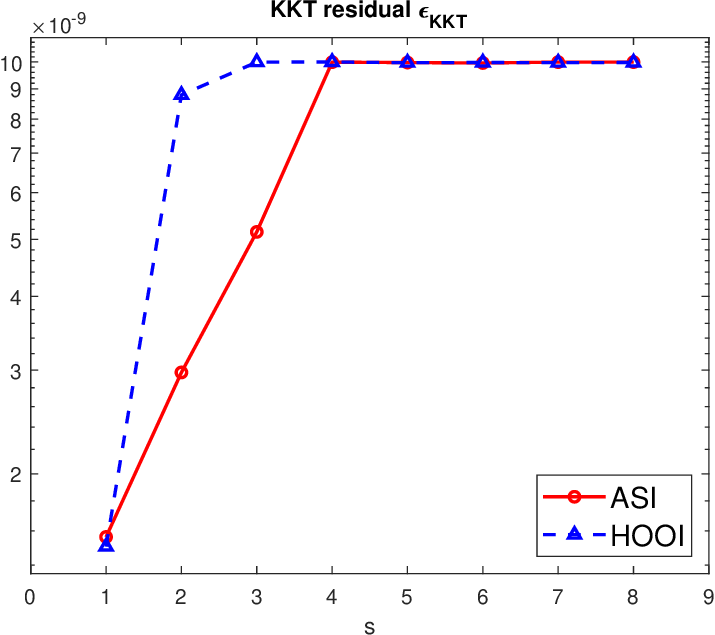}}
  &\resizebox*{0.28\textwidth}{0.15\textheight}{\includegraphics{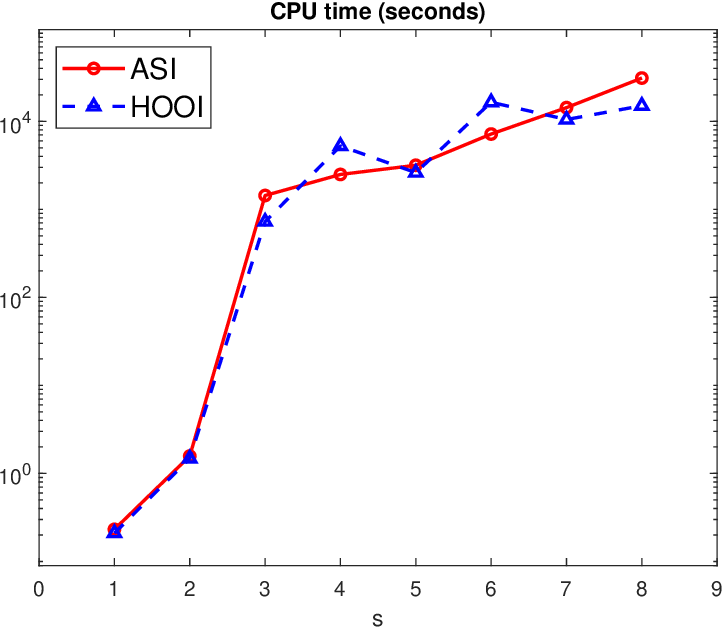}}
  &\resizebox*{0.28\textwidth}{0.15\textheight}{\includegraphics{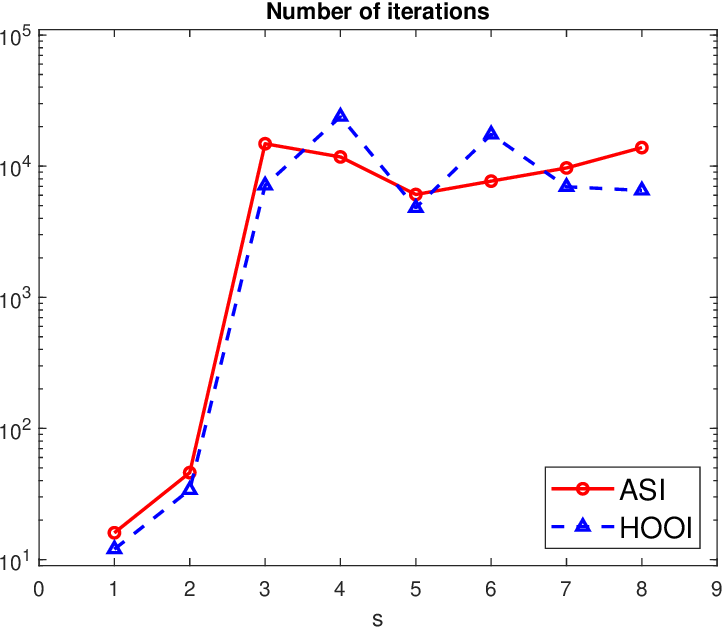}} \\
%
\rotatebox{90}{\hspace*{1.2cm}$\eta=2^{-4}$}
  &\resizebox*{0.28\textwidth}{0.15\textheight}{\includegraphics{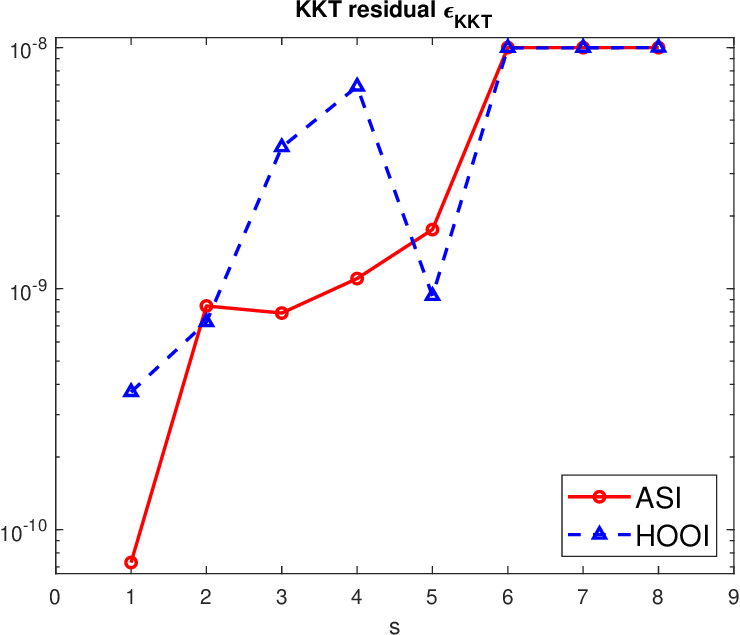}}
  &\resizebox*{0.28\textwidth}{0.15\textheight}{\includegraphics{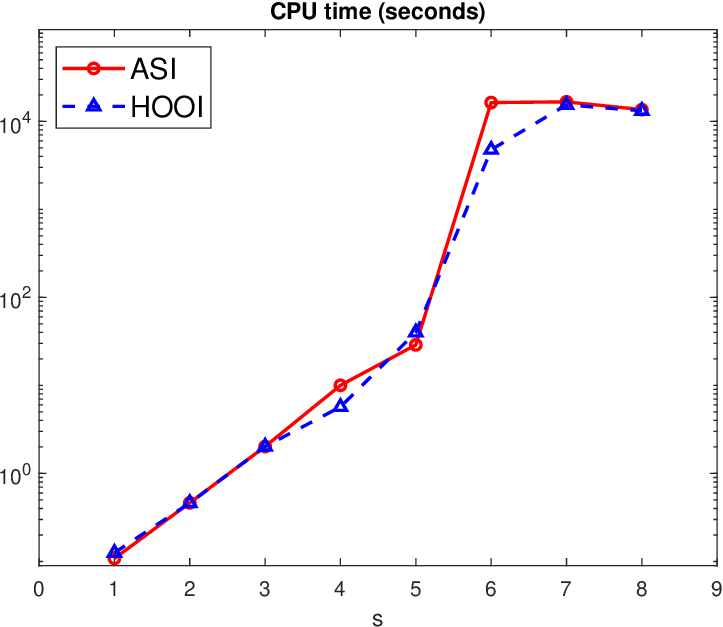}}
  &\resizebox*{0.28\textwidth}{0.15\textheight}{\includegraphics{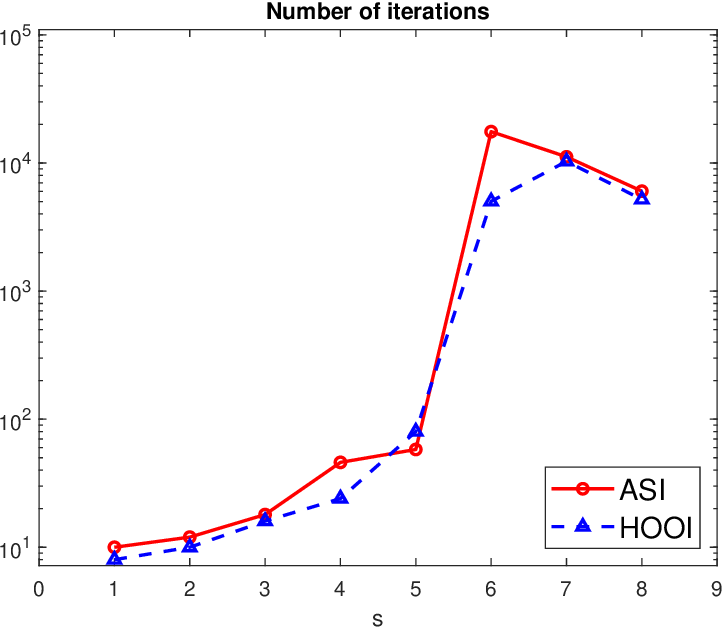}} \\
\rotatebox{90}{\hspace*{1.2cm}$\eta=2^{-5}$}
  &\resizebox*{0.28\textwidth}{0.15\textheight}{\includegraphics{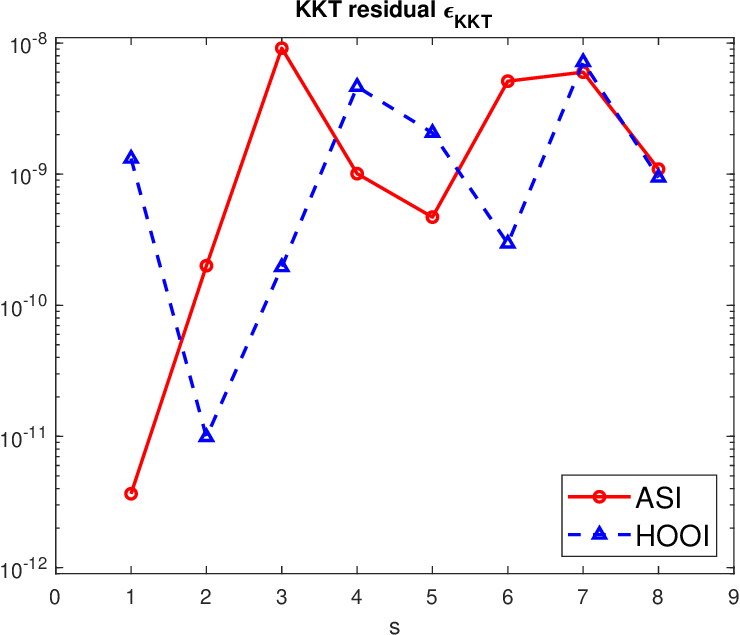}}
  &\resizebox*{0.28\textwidth}{0.15\textheight}{\includegraphics{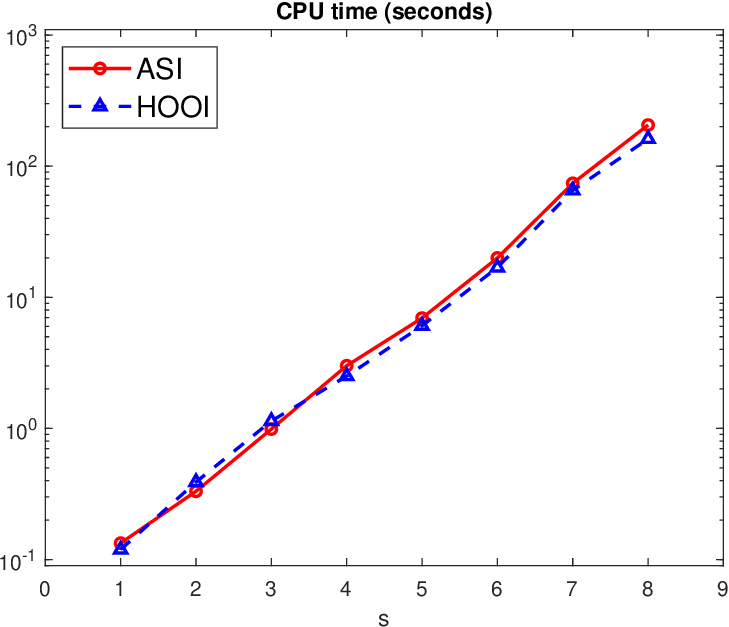}}
  &\resizebox*{0.28\textwidth}{0.15\textheight}{\includegraphics{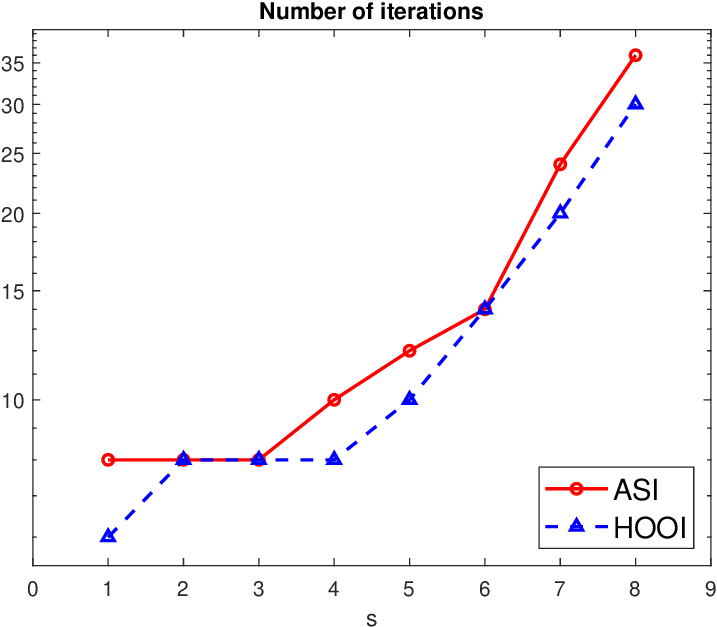}}
\end{tabular}\par
}
\vspace{-0.15 cm}
\caption{\small Scalability of HOOI and ASI on complex tensors with $[n_1,n_2,n_3]$ varying according to \eqref{eq:tensor-sizes}
    for $1\le s\le 8$ and $[k_1,k_2,k_3]=[12,11,10]$. {\em Left panel:\/} KKT residual $\tilde\epsilon_{\KKT}$,
    {\em Middle panel:\/} CPU time, and {\em Right panel:\/} the number of iterations.
  }
\label{fig:scale-cmpx}
\end{figure}

\section{Concluding Remarks}\label{sec:conclu}
We have performed detailed convergence analysis on two popular methods, the higher-order orthogonal iteration (HOOI) \cite{dldv:2000b} and the alternating subspace iteration (ASI) \cite{krdl:1980}, for computing the Tucker decomposition
of a multiple-mode tensor. Previous convergence analysis on ASI was done by
Kroonenberg and De Leeuw~\cite{krdl:1980}  but their analysis has gaps, as we discussed in
\Cref{rk:ASI-cvg}. Also our \Cref{thm:cvg4ASI} contains results that are much deeper than what were concluded in \cite{krdl:1980}.
Xu~\cite{xu:2018} performed convergence analysis on the so-called greedy HOOI, which adds an extra alignment on $P_{\ell}^{(j+1)}$
every time it is computed.
Xu's analysis relies on the Kurdyka–Lojasiewicz theory
for semi-algebraic functions.
Our analysis is more accessible to the numerical linear algebra community, self-contained, and
it treats both real and complex tensors without distinction.

Both \Cref{alg:HOOI,alg:ASI} use the Gauss-Seidel-type updating. Although not detailed,
the Jacobi-type updating works, too, especially for the case when tensor $B$ nearly fits the TD model.
In the Jacobi-type updating, at iteration $j$, $C_{\ell}(\cdots)$ for
$1\le\ell\le m$ can be evaluated at $(P_1^{(j)},P_2^{(j)},\ldots,P_m^{(j)})$ in paralell, instead of
sequentially at $(P_1^{(j+1)},\ldots,P_{\ell-1}^{(j+1)},P_{\ell}^{(j)}\ldots,P_m^{(j)})$. We experimented
with the Jacobi-type updating and found that it still works but for very small $\eta$ (say $\eta\le 2^{-5}$ or smaller).
Both \Cref{alg:HOOI,alg:ASI} can be accelerated by employing the LOCG (locally optimal conjugate gradient) technique.
We have also done experments on that and found HOOI combined with LOCG did not perform better, but ASI with LOCG did.
The latter is a special case of what is in \cite{liwy:2026tbd:arXiv}, to which the reader is referred
 for detail.

{\small
\def\noopsort#1{}\def\l{\char32l}\def\v#1{{\accent20 #1}}
  \let\^^_=\v\def\hbk{hardback}\def\pbk{paperback}

}

\end{document}